\documentclass[12pt,a4paper]{article}

\usepackage[verbose]{geometry}
\input{PreludeA}

\usepackage{titling}
\usepackage{fancyhdr}

\hangsecnum

\title{The heart of a combinatorial model category}
\author{Zhen~Lin Low}
\date{6 January 2015}

\begin{document}

\maketitle
\footpar{Department of Pure Mathematics and Mathematical Statistics, University of Cambridge, Cambridge, UK. \textsc{E-mail address}: \texttt{Z.L.Low@dpmms.cam.ac.uk}}

\begin{abstract}
We show that every small model category that satisfies certain size conditions can be completed to yield a combinatorial model category, and conversely, every combinatorial model category arises in this way.
We will also see that these constructions preserve right properness and compatibility with simplicial enrichment.
Along the way, we establish some technical results on the index of accessibility of various constructions on accessible categories, which may be of independent interest.
\end{abstract}

\section*{Introduction}

Category-theoretic homotopy theory has seen a boom in recent decades.
One development was the introduction of the notion of `combinatorial model categories' by \citet{Smith:1998}.
These correspond to what \citet{HTT} calls `presentable $\infty$-categories' and are therefore a homotopy-theoretic generalisation of the locally presentable categories of \citet{Gabriel-Ulmer:1971}.
The classification of locally $\kappa$-presentable categories says that each one is equivalent to the free $\kappa$-ind-completion of a $\kappa$-cocomplete small category, and Lurie proved the analogous proposition for presentable $\infty$-categories, so it should at least seem plausible that every combinatorial model category is generated by a small model category in an appropriate sense.

Indeed, the work of \citet{Beke:2000} suggests that more should be true.
As stated in the abstract of \opcit,
\begin{quote}
If a Quillen model category can be specified using a certain logical syntax (intuitively, `is algebraic/combinatorial enough'), so that it can be defined in any category of sheaves, then the satisfaction of Quillen's axioms over any site is a purely formal consequence of their being satisfied over the category of sets.
\end{quote}
In the same vein, we can show that the answer to the question of whether a set of generating cofibrations and trivial cofibrations in a locally presentable category really do generate a combinatorial model category depends only on an essentially small full subcategory of small objects, which we may think of as an analogue of the Löwenheim--Skolem theorem in logic.
More precisely:

\begin{thm*}
Let $\mathcal{M}$ be a locally presentable category and let $\mathcal{I}$ and $\mathcal{I}'$ be subsets of $\mor \mathcal{M}$.
There is a regular cardinal $\lambda$ such that the weak factorisation systems cofibrantly generated by $\mathcal{I}$ and $\mathcal{I}'$ underlie a model structure on $\mathcal{M}$ if and only if their restrictions to $\Kompakt[\lambda]{\mathcal{M}}$ underlie a model structure on $\Kompakt[\lambda]{\mathcal{M}}$, where $\Kompakt[\lambda]{\mathcal{M}}$ is the full subcategory of $\lambda$-presentable objects in $\mathcal{M}$.
\end{thm*}

The main difficulty is in choosing a definition of `weak equivalence in $\mathcal{M}$' for which we can verify the model category axioms.
As it turns out, what works is to define `weak equivalence' to be a morphism such that the right half of its (trivial cofibration, fibration)-factorisation is a trivial fibration.
This allows us to apply the theory of accessible categories: the key result needed is a special case of the well-known theorem of \citet[\Sect 5.1]{Makkai-Pare:1989} concerning weighted 2-limits of diagrams of accessible categories.
Moreover, by using good estimates for the index of accessibility of the categories obtained in this way, we can establish a stronger result:

\begin{thm*}
Let $\mathcal{M}$ be a locally presentable category and let $\mathcal{I}$ and $\mathcal{I}'$ be subsets of $\mor \mathcal{M}$.
Suppose $\kappa$ and $\lambda$ are regular cardinals that satisfy the following hypotheses:
\begin{itemize}
\item $\mathcal{M}$ is a locally $\kappa$-presentable category, and $\kappa$ is sharply less than $\lambda$.

\item $\Kompakt[\lambda]{\mathcal{M}}$ is closed under finite limits in $\mathcal{M}$.

\item There are $< \lambda$ morphisms between any two $\kappa$-presentable objects in $\mathcal{M}$.

\item $\mathcal{I}$ and $\mathcal{I}'$ are $\lambda$-small sets of morphisms between $\kappa$-presentable objects.
\end{itemize}
Then the weak factorisation systems cofibrantly generated by $\mathcal{I}$ and $\mathcal{I}'$ underlie a model structure on $\mathcal{M}$ if and only if their restrictions to $\Kompakt[\lambda]{\mathcal{M}}$ underlie a model structure on $\Kompakt[\lambda]{\mathcal{M}}$.
\end{thm*}

This is essentially what \autoref{thm:completeness.for.strongly.accessible.model.categories} states.
Moreover, given $\mathcal{M}$, $\mathcal{I}$, and $\mathcal{I}'$, we can always find regular cardinals $\kappa$ and $\lambda$ satisfying the hypotheses above.
Thus, if $\mathcal{M}$ is a combinatorial model category, there is a regular cardinal $\lambda$ such that $\Kompakt[\lambda]{\mathcal{M}}$ not only inherits a model structure from $\mathcal{M}$ but also determines $\mathcal{M}$ as a combinatorial model category \dash the subcategory $\Kompakt[\lambda]{\mathcal{M}}$ might be called the `heart' of $\mathcal{M}$.
(For details, see \autoref{prop:heart.of.a.combinatorial.model.category}.)
When we have explicit sets of generating cofibrations and generating trivial cofibrations, we can also give explicit $\kappa$ and $\lambda$ for which this happens: 
\begin{itemize}
\item If $\mathcal{M}$ is the category of simplicial sets with the Kan--Quillen model structure, then we can take $\kappa = \aleph_0$ and $\lambda = \aleph_1$.

\item If $\mathcal{M}$ is the category of unbounded chain complexes of left $R$-modules, then we can take $\kappa = \aleph_0$ and $\lambda$ to be the smallest uncountable regular cardinal such that $R$ is $\lambda$-small (as a set).

\item If $\mathcal{M}$ is the category of symmetric spectra of \citet{HSS:2000} with the stable model structure, then we can take $\kappa = \aleph_1$ and $\lambda$ to be the cardinal successor of $2^{2^{\aleph_0}}$.
\end{itemize}
In the converse direction, we obtain a sufficient condition for an essentially small model category $\mathcal{K}$ to arise in this fashion: see \autoref{thm:model.structure.for.ind-objects}.

The techniques used in the proof of the main theorem are easily generalised, allowing us to make sense of a remark of \citet{Dugger:2001c}:
\begin{quote}
[\ldots] for a combinatorial model category the interesting part of the homotopy theory is all concentrated within some small subcategory \dash beyond sufficiently large cardinals the homotopy theory is somehow ``formal''.
\end{quote}
For illustration, we will see how to validate the above heuristic in the cases of right properness and axiom {\LiningNumbers SM7}.

%
%
%
\needspace{2.5\baselineskip}
The structure of this paper is as follows:
\begin{itemize}
\item \Sect \ref{sect:presentable.objects} contains some technical results on presentable objects and filtered colimits thereof.
In particular, the definition of `sharply less than' is recalled, in preparation for the statement of the main result.

\item \Sect \ref{sect:accessible.constructions} is an analysis of some special cases of the theorem of Makkai and Paré on weighted 2-limits of accessible categories (see Theorem 5.1.6 in \citep{Makkai-Pare:1989}, or \citep[\Sect 2.H]{LPAC}), with a special emphasis on the index of accessibility of the categories and functors involved.

The results appearing in this section are related to those appearing in a preprint of \citet{Ulmer:1977} and probably well known to experts; nonetheless, for the sake of completeness, full proofs are given.

\item \Sect \ref{sect:accessibly.generated.categories} introduces the notion of accessibly generated category, which is a size-restricted analogue of the notion of accessible category.

\item \Sect \ref{sect:accessible.factorisation.systems} collects together some results about cofibrantly generated weak factorisation systems on locally presentable categories.

\item \Sect \ref{sect:strongly.combinatorial.model.categories} establishes the main result: that every combinatorial model category is generated by a small model category, and conversely, that small model categories satisfying certain size conditions generate combinatorial model categories.
\end{itemize}
This paper also includes some appendices covering background material:
\begin{itemize}
\item \Sect \ref{sect:accessibility} is an overview of the basic theory of accessible categories.
General references for this topic include Chapter 2 of \citep{LPAC}, and Chapter 5 of \citep{Borceux:1994b}.

\item \Sect \ref{sect:factorisation} sets up our notation and terminology regarding factorisation systems.

\item \Sect \ref{sect:model.structures} contains the definition of various kinds of model categories.
\end{itemize}

\subsection*{Acknowledgements}

The author is indebted to Jiří Rosický for bringing \autoref{thm:accessible.iso-comma.categories} to his attention: without this fact, it would have been impossible to control the index of accessibility of all the various subcategories considered in the proof of the main result.
Thanks are also due to David White for many helpful comments, and to Hans-E.\@ Porst \parencite*{Porst:2014} for unearthing \citep{Ulmer:1977} and drawing attention to the results contained therein.
Finally, the author is grateful to an anonymous referee for suggestions leading to a more streamlined exposition.

The author gratefully acknowledges financial support from the Cambridge Commonwealth, European and International Trust and the Department of Pure Mathematics and Mathematical Statistics.
\section{Presentable objects}
\label{sect:presentable.objects}

\begin{numpar}
Throughout this section, $\kappa$ is an arbitrary regular cardinal.
\end{numpar}

\begin{dfn}
Let $\mathcal{C}$ be a locally small category.
\begin{itemize}
\item Let $\lambda$ be a regular cardinal.
A \strong{$\tuple{\kappa, \lambda}$-presentable object} in $\mathcal{C}$ is an object $A$ in $\mathcal{C}$ such that the representable functor $\Hom[\mathcal{C}]{A}{\blank} : \mathcal{C} \to \cat{\Set}$ preserves colimits of all $\lambda$-small $\kappa$-filtered diagrams.

We write $\Kompakt[\kappa][\lambda]{\mathcal{C}}$ for the full subcategory of $\mathcal{C}$ spanned by the $\tuple{\kappa, \lambda}$-presentable objects.

\item A \strong{$\kappa$-presentable object} in $\mathcal{C}$ is an object in $\mathcal{C}$ that is $\tuple{\kappa, \lambda}$-presentable for all regular cardinals $\lambda$.

We write $\Kompakt[\kappa]{\mathcal{C}}$ for the full subcategory of $\mathcal{C}$ spanned by the $\kappa$-presentable objects.
\end{itemize}
\end{dfn}

\begin{example}
A set is $\kappa$-small if and only if it is a $\kappa$-presentable object in $\cat{\Set}$.
\end{example}

\begin{remark}
Although every $\aleph_0$-small (\ie finite) category is $\aleph_0$-presentable as an object in $\cat{\Cat}$, not every $\aleph_0$-presentable object in $\cat{\Cat}$ is $\aleph_0$-small.
The difference disappears for uncountable regular cardinals.
\end{remark}

\begin{lem}
\label{lem:small.objects}
Let $\mathcal{C}$ be a locally small category and let $B : \mathcal{D} \to \mathcal{C}$ be a $\kappa$-small diagram.
If each $B d$ is a $\tuple{\kappa, \lambda}$-presentable object in $\mathcal{C}$, then the colimit $\indlim_\mathcal{D} B$, if it exists, is also a $\tuple{\kappa, \lambda}$-presentable object in $\mathcal{C}$.
\end{lem}
\begin{proof}
This follows from the fact that $\prolim_{\op{\mathcal{D}}} \Func{\op{\mathcal{D}}}{\cat{\Set}} \to \cat{\Set}$ preserves colimits of small $\kappa$-filtered diagrams.
\end{proof}

\begin{lem}
\label{lem:back-and-forth.approximation.of.isomorphisms}
Assume the following hypotheses:
\begin{itemize}
\item $\mathcal{E}$ is a locally small category with colimits of small $\kappa$-filtered diagrams.

\item $X, Y : \mathcal{I} \to \mathcal{E}$ are two small $\lambda$-filtered diagrams whose vertices are $\lambda$-presentable objects in $\mathcal{E}$, where $\kappa \le \lambda$.

\item $\phi : X \hoto Y$ is a natural transformation.
\end{itemize}
Let $i_0$ be an object in $\mathcal{I}$.
If $\indlim_\mathcal{I} \phi : \indlim_\mathcal{I} X \to \indlim_\mathcal{I} Y$ is an isomorphism in $\mathcal{E}$, then there is a chain $I : \kappa \to \mathcal{I}$ such that $I \argp{0} = i_0$ and $\indlim_{\gamma < \kappa} \phi_{I \argp{\gamma}} : \indlim_{\gamma < \kappa} X I \argp{\gamma} \to \indlim_{\gamma < \kappa} Y I \argp{\gamma}$ is an isomorphism in $\mathcal{E}$.
\end{lem}
\begin{proof}
Let $C = \indlim_\mathcal{I} X$ and $D = \indlim_\mathcal{I} Y$, let $c_i : X i \to C$ and $d_i : Y i \to D$ are the components of the respective colimiting cocones and let $e = \indlim_\mathcal{I} \phi$.
We will construct $I : \kappa \to \mathcal{I}$ by transfinite induction.
\begin{itemize}
\item Let $I \argp{0} = i_0$.

\item Given an ordinal $\alpha < \kappa$ and an object $I \argp{\alpha}$ in $\mathcal{I}$, choose an object $I \argp{\alpha+1}$ in $\mathcal{I}$ and a morphism $I \argp{\alpha \to \alpha+1} : I \argp{\alpha} \to I \argp{\alpha+1}$ in $\mathcal{I}$ for which there is a morphism $Y I \argp{\alpha} \to X I \argp{\alpha+1}$ making the diagram in $\mathcal{E}$ shown below commute:
\[
\begin{tikzcd}[column sep=9.0ex]
X I \argp{\alpha} \dar[swap]{\phi_{I \argp{\alpha}}} \rar{X I \argp{\alpha \to \alpha+1}} &
X I \argp{\alpha+1} \dar{\phi_{I \argp{\alpha+1}}} \rar{c_{I \argp{\alpha+1}}} \rar &
C \dar{e} \\
Y I \argp{\alpha} \urar[dashed]{\psi_\alpha} \rar[swap]{Y I \argp{\alpha \to \alpha+1}} &
Y I \argp{\alpha+1} \rar[swap]{d_{I \argp{\alpha+1}}} &
D
\end{tikzcd}
\]
Such a choice exists: since $Y I \argp{\alpha}$ is a $\lambda$-presentable object in $\mathcal{E}$ and $\mathcal{I}$ is $\lambda$-filtered, there is an object $i'$ in $\mathcal{I}$ and a commutative diagram in $\mathcal{E}$ of the form below,
\[
\begin{tikzcd}
X I \argp{\alpha} \arrow{dd}[swap]{\phi_{I \argp{\alpha}}} \drar[dashed]{s} \arrow{rr}{c_{I \argp{\alpha}}} &&
C \dar[equals] \\
&
X i' \rar{c_{i'}} \rar &
C \dar[leftarrow]{\inv{e}} \\
Y I \argp{\alpha} \urar[dashed, swap]{t} \arrow{rr}[swap]{d_{I \argp{\alpha}}} &&
D
\end{tikzcd}
\]
so there exist an object $i''$ in $\mathcal{I}$ and morphisms $u : I \argp{\alpha} \to i''$ and $v : i' \to i''$ such that the following diagram in $\mathcal{E}$ commutes,
\[
\begin{tikzcd}
X I \argp{\alpha} \drar[swap]{s} \arrow{rr}{X u} &&
X i'' \\
&
X i' \urar[swap]{X v}
\end{tikzcd}
\]
and similarly, there exist an object $I \argp{\alpha+1}$ in $\mathcal{I}$ and a morphism $w : i'' \to I \argp{\alpha+1}$ in $\mathcal{I}$ such that the diagram in $\mathcal{E}$ shown below commutes,
\[
\begin{tikzcd}
{} &
X i' \rar{X v} &
X i'' \rar{X w} &
X I \argp{\alpha+1} \dar{\phi_{I \argp{\alpha+1}}} \\
Y I \argp{\alpha} \urar{t} \arrow{rr}[swap]{Y u} &&
Y i'' \rar[swap]{Y w} &
Y I \argp{\alpha+1}
\end{tikzcd}
\]
so we may take $\psi_\alpha : Y I \argp{\alpha} \to X I \argp{\alpha+1}$ to be the composite $X w \circ X v \circ t$ and $I \argp{\alpha \to \alpha+1} : I \argp{\alpha} \to I \argp{\alpha+1}$ to be the composite $w \circ u$.

\item Given a limit ordinal $\beta < \kappa$, assuming $I$ is defined on the ordinals $\alpha < \beta$, define $I \argp{\beta}$ and $I \argp{\alpha \to \beta}$ (for $\alpha < \beta$) by choosing a cocone over the given $\alpha$-chain in $\mathcal{I}$.
\end{itemize}
The above yields a chain $I : \kappa \to \mathcal{I}$.
By construction, for every ordinal $\alpha < \kappa$, the following diagram in $\mathcal{E}$ commutes,
\[
\begin{tikzcd}
X I \argp{\alpha} \dar[swap]{\phi_{I \argp{\alpha}}} \rar &
\indlim_{\gamma < \kappa} X I \argp{\gamma} \dar{\indlim_{\gamma < \kappa} \phi_{I \argp{\gamma}}} \\
Y I \argp{\alpha} \dar[swap]{\psi_\alpha} \rar &
\indlim_{\gamma < \kappa} Y I \argp{\gamma} \dar{\indlim_{\gamma < \kappa} \psi_\gamma} \\
X I \argp{\alpha+1} \rar &
\indlim_{\gamma < \kappa} X I \argp{\gamma+1}
\end{tikzcd}
\]
where the horizontal arrows are the respective colimiting cocone components.
The composite of the left column is $X I \argp{\alpha \to \alpha+1} : X I \argp{\alpha} \to X I \argp{\alpha+1}$, so $\indlim_{\gamma < \kappa} \phi_{I \argp{\gamma}} : \indlim_{\gamma < \kappa} X I \argp{\gamma} \to \indlim_{\gamma < \kappa} Y I \argp{\gamma}$ is a split monomorphism in $\mathcal{E}$.
Similarly, the diagram below commutes,
\[
\begin{tikzcd}
Y I \argp{\alpha} \dar[swap]{\psi_\alpha} \rar &
\indlim_{\gamma < \kappa} Y I \argp{\gamma} \dar{\indlim_{\gamma < \kappa} \psi_\gamma} \\
X I \argp{\alpha+1} \dar[swap]{\phi_{I \argp{\alpha+1}}} \rar &
\indlim_{\gamma < \kappa} X I \argp{\gamma} \dar{\indlim_{\gamma < \kappa} \phi_{I \argp{\gamma}}} \\
Y I \argp{\alpha+1} \rar &
\indlim_{\gamma < \kappa} Y I \argp{\gamma+1}
\end{tikzcd}
\]
so $\indlim_{\gamma < \kappa} \phi_{I \argp{\gamma}} : \indlim_{\gamma < \kappa} X I \argp{\gamma} \to \indlim_{\gamma < \kappa} Y I \argp{\gamma}$ is also a split epimorphism in $\mathcal{E}$.
Thus, $I : \kappa \to \mathcal{I}$ is the desired chain.
\end{proof}

The following notion is due to \citet{Makkai-Pare:1989}.

\begin{dfn}
Let $\kappa$ and $\lambda$ be regular cardinals.
We write `$\kappa \sharplylt \lambda$' and we say `$\kappa$ is \strong{sharply less than} $\lambda$' for the following condition:
\begin{itemize}
\item $\kappa < \lambda$ and, for all $\lambda$-small sets $X$, there is a $\lambda$-small cofinal subset of $\powerset[\kappa]{X}$, the set of all $\kappa$-small subsets of $X$ (partially ordered by inclusion).
\end{itemize}
\end{dfn}

\begin{example}
If $\lambda$ is an uncountable regular cardinal, then $\aleph_0 \sharplylt \lambda$: indeed, for any $\lambda$-small set $X$, the set $\powerset[\aleph_0]{X}$ itself is $\lambda$-small.
\end{example}

\begin{example}
If $\lambda$ is a strongly inaccessible cardinal and $\kappa < \lambda$, then $\kappa \sharplylt \lambda$: indeed, for any $\lambda$-small set $X$, the set $\powerset[\kappa]{X}$ itself is $\lambda$-small.
\end{example}

\begin{example}
Let $\kappa^+$ be the cardinal successor of $\kappa$.
Then $\kappa \sharplylt \kappa^+$: every $\kappa^+$-small set can be mapped bijectively onto an initial segment $\alpha$ of $\kappa$ (but possibly all of $\kappa$), and it is clear that the subposet
\[
\set{ \beta }{ \beta \le \alpha } \subseteq \powerset[\kappa]{\alpha}
\]
is a $\kappa^+$-small cofinal subposet of $\powerset[\kappa]{\alpha}$: given any $\kappa$-small subset $X \subseteq \alpha$, we must have $\sup X \le \alpha$, and $X \subseteq \sup X$ by definition.
\end{example}

The following is a partial converse to \autoref{lem:small.objects}.

\begin{prop}
\label{prop:presentable.objects.are.presentable}
Let $\mathcal{C}$ be a $\kappa$-accessible category.
If $\lambda$ is a regular cardinal and $\kappa \sharplylt \lambda$, then the following are equivalent for an object $C$ in $\mathcal{C}$:
\begin{enumerate}[(i)]
\item $C$ is a $\lambda$-presentable object in $\mathcal{C}$.

\item There is a $\lambda$-small $\kappa$-filtered diagram $A : \mathcal{J} \to \mathcal{C}$ such that each $A j$ is a $\kappa$-presentable object in $\mathcal{C}$ and $C \cong \indlim_\mathcal{J} A$.

\item There is a $\lambda$-small $\kappa$-directed diagram $A : \mathcal{J} \to \mathcal{C}$ such that each $A j$ is a $\kappa$-presentable object in $\mathcal{C}$ and $C$ is a retract of $\indlim_\mathcal{J} A$.
\end{enumerate}
\end{prop}
\begin{proof} \openproof
(i) \iff (ii).
See Proposition 2.3.11 in \citep{Makkai-Pare:1989}.

\medskip\noindent
(i) \iff (iii).
See the proof of Theorem~2.3.10 in \citep{Makkai-Pare:1989} or Remark 2.15 in \citep{LPAC}.
\end{proof}

\begin{lem}
\label{lem:estimate.of.morphisms.in.accessible.categories}
Let $\mathcal{C}$ be a $\kappa$-accessible category, let $A$ be a $\kappa$-presentable object in $\mathcal{C}$, and let $B$ be a $\lambda$-presentable object in $\mathcal{C}$.
If the hom-set $\Hom[\mathcal{C}]{A}{A'}$ is $\mu$-small for all $\kappa$-presentable objects $A'$ in $\mathcal{C}$ and $\kappa \sharplylt \lambda$, then the hom-set $\Hom[\mathcal{C}]{A}{B}$ has cardinality $< \max \set{ \lambda, \mu }$.
\end{lem}
\begin{proof}
By \autoref{prop:presentable.objects.are.presentable}, there is a $\lambda$-small $\kappa$-filtered diagram $Y : \mathcal{J} \to \mathcal{C}$ such that each $Y j$ is a $\kappa$-presentable object in $\mathcal{C}$ and $B$ is a retract of $\indlim_\mathcal{J} Y$.
Since $A$ is a $\kappa$-presentable object in $\mathcal{C}$, we have
\[
\Hom[\mathcal{C}]{A}{\indlim_\mathcal{J} Y} \cong \textstyle \indlim_\mathcal{J} \Hom[\mathcal{C}]{A}{Y}
\]
and the RHS is a set of cardinality $< \max \set{ \lambda, \mu }$ by \autoref{lem:small.objects}; but $\Hom[\mathcal{C}]{A}{B}$ is a retract of the LHS, so we are done.
\end{proof}
\section{Accessible constructions}
\label{sect:accessible.constructions}

\begin{numpar}
Throughout this section, $\kappa$ is an arbitrary regular cardinal.
\end{numpar}

\begin{dfn}
A \strong{strongly $\kappa$-accessible functor} is a functor $F : \mathcal{C} \to \mathcal{D}$ with the following properties:
\begin{itemize}
\item Both $\mathcal{C}$ and $\mathcal{D}$ are $\kappa$-accessible categories.

\item $F$ preserves colimits of small $\kappa$-filtered diagrams.

\item $F$ sends $\kappa$-presentable objects in $\mathcal{C}$ to $\kappa$-presentable objects in $\mathcal{D}$.
\end{itemize}
\end{dfn}

\begin{example}
Given any functor $F : \mathcal{A} \to \mathcal{B}$, if $\mathcal{A}$ and $\mathcal{B}$ are essentially small categories, then the induced functor $\Ind[\kappa]{F} : \Ind[\kappa]{\mathcal{A}} \to \Ind[\kappa]{\mathcal{B}}$ is strongly $\kappa$-accessible.
If $\mathcal{B}$ is also idempotent-complete, then every strongly $\kappa$-accessible functor $\Ind[\kappa]{\mathcal{A}} \to \Ind[\kappa]{\mathcal{B}}$ is of this form (up to isomorphism).
\end{example}

\begin{prop}[Products of accessible categories]
\label{prop:products.of.accessible.categories}
If $\seq{\mathcal{C}_i}{i \in I}$ is a $\kappa$-small family of $\kappa$-accessible categories, then:
\begin{enumerate}[(i)]
\item The product $\mathcal{C} = \prod_{i \in I} \mathcal{C}_i$ is a $\kappa$-accessible category.

\item Moreover, the projection functors $\mathcal{C} \to \mathcal{C}_i$ are strongly $\kappa$-accessible functors.
\end{enumerate}
\end{prop}
\begin{proof}
It is clear that $\mathcal{C}$ has colimits of small $\kappa$-filtered diagrams: indeed, they can be computed componentwise.
Since ${\prod} : \cat{\Set}^I \to \cat{\Set}$ preserves colimits of small $\kappa$-filtered diagrams, an object in $\mathcal{C}$ is $\kappa$-presentable as soon as its components are $\kappa$-presentable objects in their respective categories.
Recalling \autoref{lem:products.of.filtered.categories}, it follows that $\mathcal{C}$ is generated under small $\kappa$-filtered colimits by a small family of $\kappa$-presentable objects, as required of a $\kappa$-accessible category.
\end{proof}

\begin{lem}
\label{lem:accessible.functors.are.strongly.accessible}
Let $\mathcal{C}$ and $\mathcal{D}$ be accessible categories and let $F : \mathcal{C} \to \mathcal{D}$ be a $\kappa$-accessible functor.
\begin{enumerate}[(i)]
\item There is a regular cardinal $\lambda$ such that $F$ is a strongly $\lambda$-accessible functor.

\item Moreover, if $\mu$ is a regular cardinal such that $\kappa \sharplylt \mu$ and $\lambda \le \mu$, then $F$ also sends $\mu$-presentable objects in $\mathcal{C}$ to $\mu$-presentable objects in $\mathcal{D}$.
\end{enumerate}
\end{lem}
\begin{proof} \openproof
(i).
See Theorem~2.19 in \citep{LPAC}.

\medskip\noindent
(ii).
Apply \autoref{lem:small.objects} and \autoref{prop:presentable.objects.are.presentable}.
\end{proof}

\begin{prop}
\label{prop:locally.presentable.functor.categories}
If $\mathcal{C}$ is a locally $\kappa$-presentable category and $\mathcal{D}$ is any small category, then the functor category $\Func{\mathcal{D}}{\mathcal{C}}$ is also a locally $\kappa$-presentable category.
\end{prop}
\begin{proof} \openproof
See Corollary~1.54 in \citep{LPAC}.
\end{proof}

\begin{prop}
\label{prop:presentable.objects.in.diagram.categories}
Let $\mathcal{C}$ be a locally small category and let $\mathcal{D}$ be a $\kappa$-small category.
\begin{enumerate}[(i)]
\item If $\lambda$ is a regular cardinal $\ge \kappa$ such that $\mathcal{C}$ has colimits of small $\lambda$-filtered diagrams and $A : \mathcal{D} \to \mathcal{C}$ is a diagram whose vertices are $\lambda$-presentable objects in $\mathcal{C}$, then $A$ is a $\lambda$-presentable object in $\Func{\mathcal{D}}{\mathcal{C}}$.

\item If $\mathcal{C}$ is a $\lambda$-accessible category and has products for $\kappa$-small families of objects, then every $\lambda$-presentable object in $\Func{\mathcal{D}}{\mathcal{C}}$ is componentwise $\lambda$-presentable.
\end{enumerate}
\end{prop}
\begin{proof} \openproof
See (the proof of) Proposition~2.23 in \citep{Low:2013a}.
\end{proof}

\begin{dfn}
Given a regular cardinal $\kappa$, a \strong{$\kappa$-accessible subcategory} of a $\kappa$-accessible category $\mathcal{C}$ is a subcategory $\mathcal{B} \subseteq \mathcal{C}$ such that $\mathcal{B}$ is a $\kappa$-accessible category and the inclusion $\mathcal{B} \embedinto \mathcal{C}$ is a $\kappa$-accessible functor.
\end{dfn}

\begin{prop}
\label{prop:strongly.accessible.subcategories}
Let $\mathcal{C}$ be a $\kappa$-accessible category and let $\mathcal{B}$ be a replete and full $\kappa$-accessible subcategory of $\mathcal{C}$.
\begin{enumerate}[(i)]
\item If $A$ is a $\kappa$-presentable object in $\mathcal{C}$ and $A$ is in $\mathcal{B}$, then $A$ is also a $\kappa$-presentable object in $\mathcal{B}$.

\item If the inclusion $\mathcal{B} \embedinto \mathcal{C}$ is strongly $\kappa$-accessible, then $\Kompakt[\kappa]{\mathcal{B}} = \mathcal{B} \cap \Kompakt[\kappa]{\mathcal{C}}$.
\end{enumerate}
\end{prop}
\begin{proof}
(i).
This is clear, since hom-sets and colimits of small $\kappa$-filtered diagrams in $\mathcal{B}$ are computed as in $\mathcal{C}$.

\medskip\noindent
(ii).
Given (i), it suffices to show that every $\kappa$-presentable object in $\mathcal{B}$ is also $\kappa$-presentable in $\mathcal{C}$, but this is precisely the hypothesis that the inclusion $\mathcal{B} \embedinto \mathcal{C}$ is strongly $\kappa$-accessible.
\end{proof}

\begin{lem}
\label{lem:strongly.accessible.subcategory.criteria}
Let $\mathcal{C}$ be a $\kappa$-accessible category and let $\mathcal{B}$ be a full subcategory of $\mathcal{C}$.
Assuming $\mathcal{B}$ is closed in $\mathcal{C}$ under colimits of small $\kappa$-filtered diagrams, the following are equivalent:
\begin{enumerate}[(i)]
\item The inclusion $\mathcal{B} \embedinto \mathcal{C}$ is a strongly $\kappa$-accessible functor.

\item Given a morphism $f : X \to Y$ in $\mathcal{C}$, if $X$ is a $\kappa$-presentable object in $\mathcal{C}$ and $Y$ is an object in $\mathcal{B}$, then $f : X \to Y$ factors through an object in $\mathcal{B}$ that is $\kappa$-presentable as an object in $\mathcal{C}$.
\end{enumerate}
\end{lem}
\begin{proof}
(i) \implies (ii). 
Let $f : X \to Y$ be a morphism in $\mathcal{C}$.
The hypothesis implies that $Y$ is a colimit in $\mathcal{C}$ of a small $\kappa$-filtered diagram in $\mathcal{B} \cap \Kompakt[\kappa]{\mathcal{C}}$; but $X$ is a $\kappa$-presentable object in $\mathcal{C}$, so $f : X \to Y$ must factor through some component of the colimiting cocone.

\medskip\noindent
(ii) \implies (i). 
In view of \autoref{lem:small.objects} and \autoref{prop:strongly.accessible.subcategories}, it suffices to show that every object in $\mathcal{B}$ is a colimit (in $\mathcal{C}$) of an essentially small $\kappa$-filtered diagram in $\mathcal{B} \cap \Kompakt[\kappa]{\mathcal{C}}$. 

Let $Y$ be an object in $\mathcal{B}$ and let $\mathcal{J}$ be the full subcategory of the slice category $\overcat{\mathcal{C}}{Y}$ spanned by the objects $\tuple{X, f}$ where $X$ is an object in $\mathcal{B}$ that is a $\kappa$-presentable object in $\mathcal{C}$.
Clearly, $\mathcal{J}$ is a full subcategory of $\commacat{\Kompakt[\kappa]{\mathcal{C}}}{Y}$.
On the other hand, the evident projection $U : \commacat{\Kompakt[\kappa]{\mathcal{C}}}{Y} \to \mathcal{C}$ is an essentially small $\kappa$-filtered diagram and the tautological cocone $U \hoto \Delta Y$ is a colimiting cocone.
\unskip\footnote{See Proposition~2.1.5 in \citep{Makkai-Pare:1989} or Proposition~2.8 in \citep{LPAC}.}
Moreover, the hypothesis implies that $\mathcal{J}$ is a $\kappa$-filtered category and a cofinal subcategory of $\commacat{\Kompakt[\kappa]{\mathcal{C}}}{Y}$.
Thus, $Y$ is also a colimit of the diagram obtained by restricting along the inclusion $\mathcal{J} \embedinto \commacat{\Kompakt[\kappa]{\mathcal{C}}}{Y}$.
This completes the proof.
\end{proof}

\begin{prop}
\label{prop:image.of.strongly.accessible.functors}
Let $F : \mathcal{C} \to \mathcal{D}$ be a strongly $\kappa$-accessible functor and let $\mathcal{D}'$ be the full subcategory of $\mathcal{D}$ spanned by the image of $F$.
\begin{enumerate}[(i)]
\item Every object in $\mathcal{D}'$ is a colimit in $\mathcal{D}$ of some small $\kappa$-filtered diagram consisting of objects in $\mathcal{D}'$ that are $\kappa$-presentable as objects in $\mathcal{D}$.

\item Every $\kappa$-presentable object in $\mathcal{D}'$ is also $\kappa$-presentable as an object in $\mathcal{D}$.

\item If $\mathcal{D}'$ is closed under colimits of small $\kappa$-filtered diagrams in $\mathcal{D}$, then $\mathcal{D}'$ is a $\kappa$-accessible subcategory of $\mathcal{D}$.
\end{enumerate}
\end{prop}
\begin{proof}
(i).
Let $D$ be any object in $\mathcal{D}'$.
By definition, there is an object $C$ in $\mathcal{C}$ such that $D = F C$, and since $\mathcal{C}$ is a $\kappa$-accessible category, there is a small $\kappa$-filtered diagram $X : \mathcal{J} \to \mathcal{C}$ such that each $X j$ is a $\kappa$-presentable object in $\mathcal{C}$ and $C \cong \indlim_\mathcal{J} X$.
Since $F : \mathcal{C} \to \mathcal{D}$ is a strongly $\kappa$-accessible functor, each $F X j$ is a $\kappa$-presentable object in $\mathcal{D}$ and we have $D \cong \indlim_\mathcal{J} F X$.

\medskip\noindent
(ii).
Moreover, if $D$ is a $\kappa$-presentable object in $\mathcal{D}'$, then $D$ must be a retract of $F X j$ for some object $j$ in $\mathcal{J}$, and so $D$ is also $\kappa$-presentable as an object in $\mathcal{D}$.

\medskip\noindent
(iii).
Any object in $\mathcal{D}'$ that is $\kappa$-presentable as an object in $\mathcal{D}$ must be $\kappa$-presentable as an object in $\mathcal{D}'$, because $\mathcal{D}'$ is a full subcategory of $\mathcal{D}$ that is closed under colimits of small $\kappa$-filtered diagrams.
Thus, by (i), $\mathcal{D}'$ is a $\kappa$-accessible subcategory of $\mathcal{D}$.
\end{proof}

\begin{thm}[Accessibility of comma categories]
\label{thm:accessible.comma.categories}
Let $F : \mathcal{C} \to \mathcal{E}$ and $G : \mathcal{D} \to \mathcal{E}$ be $\kappa$-accessible functors.
\begin{enumerate}[(i)]
\item The comma category $\commacat{F}{G}$ has colimits of small $\kappa$-filtered diagrams, created by the projection functor $\commacat{F}{G} \to \mathcal{C} \times \mathcal{D}$.

\item Given an object $\tuple{C, D, e}$ in $\commacat{F}{G}$, if $C$ is a $\kappa$-presentable object in $\mathcal{C}$, $D$ is a $\kappa$-presentable object in $\mathcal{D}$, and $F C$ is a $\kappa$-presentable object in $\mathcal{E}$, then $\tuple{C, D, e}$ is a $\kappa$-presentable object in $\commacat{F}{G}$.

\item If both $F$ and $G$ are strongly $\kappa$-accessible functors, then $\commacat{F}{G}$ is a $\kappa$-\allowhyphens accessible category, and the projection functors $P : \commacat{F}{G} \to \mathcal{C}$ and $Q : \commacat{F}{G} \to \mathcal{D}$ are strongly $\kappa$-accessible.
\end{enumerate}
\end{thm}
\begin{proof} \openproof
See (the proof of) Theorem 2.43 in \citep{LPAC}.
\end{proof}

\begin{cor}
\label{cor:accessible.arrow.categories}
If $\mathcal{C}$ is a $\kappa$-accessible category, then so is the functor category $\Func{\mathbf{2}}{\mathcal{C}}$.
Moreover, the $\kappa$-presentable objects in $\Func{\mathbf{2}}{\mathcal{C}}$ are precisely the morphisms between $\kappa$-presentable objects in $\mathcal{C}$.
\end{cor}
\begin{proof}
The functor category $\Func{\mathbf{2}}{\mathcal{C}}$ is isomorphic to the comma category $\commacat{\mathcal{C}}{\mathcal{C}}$, and $\id : \mathcal{C} \to \mathcal{C}$ is certainly a strongly $\kappa$-accessible functor, so this is a special case of \autoref{thm:accessible.comma.categories}.
\end{proof}

\begin{thm}[Accessibility of inverters]
\label{thm:accessible.inverters}
Let $R, S : \mathcal{B} \to \mathcal{E}$ be $\kappa$-accessible functors, let $\phi : R \hoto S$ be a natural transformation, and let $\mathcal{B}'$ be the full subcategory of $\mathcal{B}$ spanned by those objects $B$ in $\mathcal{B}$ such that $\phi_B : R B \to S B$ is an isomorphism in $\mathcal{E}$.
\begin{enumerate}[(i)]
\item $\mathcal{B}'$ is closed in $\mathcal{B}$ under colimits of small $\kappa$-filtered diagrams.

\item If both $R$ and $S$ are strongly $\lambda$-accessible functors and $\kappa < \lambda$, then the inclusion $\mathcal{B}' \embedinto \mathcal{B}$ is strongly $\lambda$-accessible.
\end{enumerate}
\end{thm}
\begin{proof}
(i). Straightforward.

\medskip\noindent
(ii). By \autoref{lem:strongly.accessible.subcategory.criteria}, it suffices to verify that, for every morphism $f : B \to B'$ in $\mathcal{B}$, if $B$ is a $\lambda$-presentable object in $\mathcal{B}$ and $B'$ is in $\mathcal{B}'$, then $f : B \to B'$ factors through some $\lambda$-presentable object in $\mathcal{B}$ that is also in $\mathcal{B}'$.

Since $\mathcal{B}$ is a $\lambda$-accessible category, we may choose a small $\lambda$-filtered diagram $X : \mathcal{I} \to \mathcal{B}$ such that each $X i$ is a $\lambda$-presentable object in $\mathcal{B}$ and $\indlim_\mathcal{I} X \cong B'$.
Since $B$ is a $\lambda$-presentable object in $\mathcal{B}$, there is an object $i_0$ in $\mathcal{I}$ such that $f : B \to B'$ factors as a morphism $B \to X i_0$ in $\mathcal{B}$ followed by the colimiting cocone component $X i_0 \to B'$.
Then, by \autoref{lem:back-and-forth.approximation.of.isomorphisms}, there is a chain $I : \kappa \to \mathcal{I}$ such that $I \argp{0} = i_0$ and $\hat{B} = \indlim_{\gamma < \kappa} X I \argp{\gamma}$ is in $\mathcal{B}'$.
Moreover, since $\kappa < \lambda$, $\hat{B}$ is a $\lambda$-presentable object in $\mathcal{B}$ (by \autoref{lem:small.objects}).
We have thus obtained the required factorisation of $f : B \to B'$.
\end{proof}

The next theorem is a variation on Proposition 3.1 in \citep{Chorny-Rosicky:2012} and appears as the ``pseudopullback theorem'' in \citep{Raptis-Rosicky:2015}.
Recall that the \strong{iso-comma category} $\isocommacat{F}{G}$ for functors $F : \mathcal{C} \to \mathcal{E}$ and $G : \mathcal{D} \to \mathcal{E}$ is the full subcategory of the comma category $\commacat{F}{G}$ spanned by those objects $\tuple{C, D, e}$ where $e : F C \to G D$ is an isomorphism in $\mathcal{E}$.

\begin{thm}[Accessibility of iso-comma categories]
\label{thm:accessible.iso-comma.categories}
Let $\mathcal{C}$, $\mathcal{D}$, and $\mathcal{E}$ be categories with colimits of small $\kappa$-filtered diagrams, and let $F : \mathcal{C} \to \mathcal{E}$ and $G : \mathcal{D} \to \mathcal{E}$ be be functors that preserve colimits of small $\kappa$-filtered diagrams.
\begin{enumerate}[(i)]
\item The iso-comma category $\isocommacat{F}{G}$ has colimits of small $\kappa$-filtered diagrams, created by the projection functor $\isocommacat{F}{G} \to \mathcal{C} \times \mathcal{D}$.

\item Given an object $\tuple{C, D, e}$ in $\isocommacat{F}{G}$, if $C$ is a $\lambda$-presentable object in $\mathcal{C}$, $D$ is a $\lambda$-presentable object in $\mathcal{D}$, and $F C$ is a $\lambda$-presentable object in $\mathcal{E}$, then $\tuple{C, D, e}$ is a $\lambda$-presentable object in $\isocommacat{F}{G}$.

\item If $F$ and $G$ are strongly $\lambda$-accessible functors and $\kappa < \lambda$, then the inclusion $\isocommacat{F}{G}$ is a $\lambda$-accessible category, and the projection functors $P : \isocommacat{F}{G} \to \mathcal{C}$ and $Q : \isocommacat{F}{G} \to \mathcal{D}$ are strongly $\lambda$-accessible.
\end{enumerate}
\end{thm}
\begin{proof}
(i).
This is a straightforward consequence of the hypothesis that both $F : \mathcal{C} \to \mathcal{E}$ and $G : \mathcal{D} \to \mathcal{E}$ preserve colimits of small $\kappa$-filtered diagrams.

\medskip\noindent
(ii).
Apply \autoref{prop:strongly.accessible.subcategories} and \autoref{thm:accessible.comma.categories}.

\medskip\noindent
(iii).
By \autoref{thm:accessible.inverters}, the inclusion $\isocommacat{F}{G} \embedinto \commacat{F}{G}$ is a strongly $\lambda$-\allowhyphens accessible functor.
Since the class of strongly $\lambda$-accessible functors is closed under composition, 
it follows that the projections $P : \isocommacat{F}{G} \to \mathcal{C}$ and $Q : \isocommacat{F}{G} \to \mathcal{D}$ are also strongly $\lambda$-accessible.
\end{proof}

\begin{prop}
\label{prop:preimage.of.accessible.replete.full.subcategories}
Let $\mathcal{C}$ and $\mathcal{E}$ be categories with colimits of small $\kappa$-filtered diagrams, let $\mathcal{D}$ be a replete and full subcategory of $\mathcal{E}$ that is closed under colimits of small $\kappa$-filtered diagrams, let $F : \mathcal{C} \to \mathcal{E}$ be a functor that preserves colimits of small $\kappa$-filtered diagrams, and let $\mathcal{B}$ be the preimage of $\mathcal{D}$ under $F$, so that we have the following strict pullback diagram:
\[
\begin{tikzcd}
\mathcal{B} \dar[hookrightarrow] \rar &
\mathcal{D} \dar[hookrightarrow] \\
\mathcal{C} \rar[swap]{F} &
\mathcal{E}
\end{tikzcd}
\]
\begin{enumerate}[(i)]
\item $\mathcal{B}$ is a replete and full subcategory of $\mathcal{D}$ and is closed under colimits of small $\kappa$-filtered diagrams in $\mathcal{D}$.

\item If $F : \mathcal{C} \to \mathcal{E}$ and the inclusion $\mathcal{D} \embedinto \mathcal{E}$ are strongly $\lambda$-accessible functors and $\kappa < \lambda$, then $\mathcal{B}$ is a $\lambda$-accessible subcategory of $\mathcal{C}$, and moreover, the inclusion $\mathcal{B} \hookrightarrow \mathcal{C}$ is also strongly $\lambda$-accessible.
\end{enumerate}
\end{prop}
\begin{proof}
(i).
Straightforward.

\medskip\noindent
(ii).
Consider the iso-comma category $\isocommacat{F}{\mathcal{D}}$ and the induced comparison functor $K : \mathcal{B} \to \isocommacat{F}{\mathcal{D}}$.
It is clear that $K$ is fully faithful; but since $\mathcal{D}$ is a replete subcategory of $\mathcal{C}$, for every object $\tuple{C, D, e}$ in $\isocommacat{F}{\mathcal{D}}$, there is a canonical isomorphism $K C \to \tuple{C, D, e}$, namely the one corresponding to the following commutative diagram in $\mathcal{E}$:
\[
\begin{tikzcd}
F C \dar[swap]{\id} \rar{\id} &
F C \dar{e} \\
F C \rar[swap]{e} &
D
\end{tikzcd}
\]
Thus, $K : \mathcal{B} \to \isocommacat{F}{\mathcal{D}}$ is (half of) an equivalence of categories.
\Autoref{thm:accessible.iso-comma.categories} says the projection $P : \isocommacat{F}{\mathcal{D}} \to \mathcal{C}$ is a strongly $\lambda$-accessible functor, so we may deduce that the same is true for the inclusion $\mathcal{B} \embedinto \mathcal{C}$.
\end{proof}

%
%
%
\begin{lem}
\label{lem:strong.accessibility.and.monads}
Let $\mathcal{C}$ be a locally $\kappa$-presentable category and let $\mathbb{T} = \tuple{T, \eta, \mu}$ be a monad on $\mathcal{C}$.
If the forgetful functor $U : \mathcal{C}^\mathbb{T} \to \mathcal{C}$ is strongly $\kappa$-accessible, then so is the functor $T : \mathcal{C} \to \mathcal{C}$.
\end{lem}
\begin{proof}
\Autoref{prop:left.adjoints.and.strong.accessibility} says the free $\mathbb{T}$-algebra functor $F : \mathcal{C} \to \mathcal{C}^\mathbb{T}$ is strongly $\kappa$-accessible if the forgetful functor $U : \mathcal{C}^\mathbb{T} \to \mathcal{C}$ is $\kappa$-accessible; but $T = U F$, so $T$ is strongly $\kappa$-accessible when $U$ is.
\end{proof}

The following appears as part of Proposition~4.13 in \citep{Ulmer:1977}.

\begin{thm}[The category of algebras for a strongly accessible monad]
\label{thm:strongly.accessible.monads}
Let $\mathcal{C}$ be a locally $\lambda$-presentable category, let $\mathbb{T} = \tuple{T, \eta, \mu}$ be a monad on $\mathcal{C}$ where $T : \mathcal{C} \to \mathcal{C}$ preserves colimits of small $\kappa$-filtered diagrams, and let $\mathcal{C}^\mathbb{T}$ be the category of algebras for $\mathbb{T}$.
If $T : \mathcal{C} \to \mathcal{C}$ is a strongly $\lambda$-accessible functor and $\kappa < \lambda$, then:
\begin{enumerate}[(i)]
\item Given a coequaliser diagram in $\mathcal{C}^\mathbb{T}$ of the form below,
\[
\begin{tikzcd}
\tuple{A, \alpha}
\rar[transform canvas={yshift=0.75ex}]
\rar[transform canvas={yshift=-0.75ex}] &
\tuple{B, \beta}
\rar &
\tuple{C, \gamma}
\end{tikzcd}
\]
if $A$ and $B$ are $\lambda$-presentable objects in $\mathcal{C}$, then so is $C$.

\item Given a $\lambda$-small family $\seq{\tuple{A_i, \alpha_i}}{i \in I}$ of $\mathbb{T}$-algebras, if each $A_i$ is a $\lambda$-presentable object in $\mathcal{C}$, then so is the underlying object of the $\mathbb{T}$-algebra coproduct $\sum_{i \in I} \tuple{A_i, \alpha_i}$.

\item The forgetful functor $U : \mathcal{C}^\mathbb{T} \to \mathcal{C}$ is strongly $\lambda$-accessible.
\end{enumerate}
\end{thm}
\begin{proof}
(i).
By referring to the explicit construction of coequalisers in $\mathcal{C}^\mathbb{T}$ given in the proof of Proposition 4.3.6 in \citep{Borceux:1994b} and applying \autoref{lem:small.objects}, we see that $C$ is indeed a $\lambda$-presentable object in $\mathcal{C}$ when $A$ and $B$ are, provided $T : \mathcal{C} \to \mathcal{C}$ preserves colimits of small $\kappa$-filtered diagrams \emph{and} is strongly $\lambda$-accessible.

\medskip\noindent
(ii).
Let $F : \mathcal{C} \to \mathcal{C}^\mathbb{T}$ be a left adjoint for $U : \mathcal{C}^\mathbb{T} \to \mathcal{C}$.
In the proof of Proposition~4.3.4 in \citep{Borceux:1994b}, we find that the $\mathbb{T}$-algebra coproduct $\sum_{i \in I} \tuple{A_i, \alpha_i}$ may be computed by a coequaliser diagram of the following form:
\[
\begin{tikzcd}
F \argp{\sum_{i \in I} T A_i}
\rar[transform canvas={yshift=0.75ex}]
\rar[transform canvas={yshift=-0.75ex}] &
F \argp{\sum_{i \in I} A_i}
\rar &
\sum_{i \in I} \tuple{A_i, \alpha_i}
\end{tikzcd}
\]
Since $T : \mathcal{C} \to \mathcal{C}$ is strongly $\lambda$-accessible, the underlying objects of the $\mathbb{T}$-\allowhyphens algebras $F \argp{\sum_{i \in I} T A_i}$ and $F \argp{\sum_{i \in I} A_i}$ are $\lambda$-presentable objects in $\mathcal{C}$.
Thus, by (i), the underlying object of $\sum_{i \in I} \tuple{A_i, \alpha_i}$ must also be a $\lambda$-presentable object in $\mathcal{C}$.

\medskip\noindent
(iii).
It is shown in the proof of Theorem 5.5.9 in \citep{Borceux:1994b} that the full subcategory $\mathcal{F}$ of $\mathcal{C}^\mathbb{T}$ spanned by the image of $\Kompakt[\lambda]{\mathcal{C}}$ under $F : \mathcal{C} \to \mathcal{C}^\mathbb{T}$ is a dense subcategory.
Let $\mathcal{G}$ be the smallest replete full subcategory of $\mathcal{C}^\mathbb{T}$ that contains $\mathcal{F}$ and is closed under colimits of $\lambda$-small diagrams in $\mathcal{C}$.
Observe that (i) and (ii) imply that the underlying object of every $\mathbb{T}$-algebra that is in $\mathcal{G}$ must be a $\lambda$-presentable object in $\mathcal{C}$.
To show that the forgetful functor $U : \mathcal{C}^\mathbb{T} \to \mathcal{C}$ is strongly $\lambda$-accessible, it is enough to verify that every $\lambda$-presentable object in $\mathcal{C}^\mathbb{T}$ is in $\mathcal{G}$.

It is not hard to see that the comma category $\commacat{\mathcal{G}}{\tuple{A, \alpha}}$ is an essentially small $\lambda$-filtered category for any $\mathbb{T}$-algebra $\tuple{A, \alpha}$, and moreover, it can be shown that the tautological cocone for the canonical diagram $\commacat{\mathcal{G}}{\tuple{A, \alpha}} \to \mathcal{C}^\mathbb{T}$ is a colimiting cocone.
Thus, if $\tuple{A, \alpha}$ is a $\lambda$-presentable object in $\mathcal{C}^\mathbb{T}$, it must be a retract of an object in $\mathcal{G}$.
But $\mathcal{G}$ is closed under retracts, so $\tuple{A, \alpha}$ is indeed in $\mathcal{G}$.
\end{proof}
%
%
%

The following result on the existence of free algebras for a pointed endofunctor is a special case of a general construction due to \citet{Kelly:1980}.

\begin{thm}[Free algebras for a pointed endofunctor]
\label{thm:free.algebra.sequence}
Let $\mathcal{C}$ be a category with joint coequalisers for $\kappa$-small families of parallel pairs and colimits of chains of length $\le \kappa$, let $\tuple{J, \iota}$ be a pointed endofunctor on $\mathcal{C}$ such that $J : \mathcal{C} \to \mathcal{C}$ preserves colimits of $\kappa$-chains, and let $\mathcal{C}^{\tuple{J, \iota}}$ be the category of algebras for $\tuple{J, \iota}$.
\begin{enumerate}[(i)]
\item The forgetful functor $U : \mathcal{C}^{\tuple{J, \iota}} \to \mathcal{C}$ has a left adjoint, say $F : \mathcal{C} \to \mathcal{C}^{\tuple{J, \iota}}$.

\item Let $\lambda$ be a regular cardinal.
If $J : \mathcal{C} \to \mathcal{C}$ sends $\lambda$-presentable objects to $\lambda$-presentable objects and $\kappa < \lambda$, then the functor $U F : \mathcal{C} \to \mathcal{C}$ has the same property.
\end{enumerate}
\end{thm}
\begin{proof}
Let $X$ be an object in $\mathcal{C}$.
We define an object $X_\alpha$ for each ordinal $\alpha \le \kappa$, a morphism $q_\alpha : J X_\alpha \to X_{\alpha + 1}$ for each ordinal $\alpha < \kappa$, and a morphism $s_{\alpha, \beta} : X_\alpha \to X_\beta$ for each pair $\tuple{\alpha, \beta}$ of ordinals such that $\alpha \le \beta \le \kappa$ by transfinite recursion as follows:
\begin{itemize}
\item We define $X_0 = X$ and $s_{0,0} = \id_{X_0}$.

\item For each ordinal $\beta < \kappa$, given $X_\alpha$ for all $\alpha \le \beta$, $q_\alpha$ for all $\alpha < \beta$, and $s_{\alpha, \beta}$ for all $\alpha \le \beta$, we define $q_\beta : J X_\beta \to X_{\beta + 1}$ to be the joint coequaliser of the parallel pairs
\[
\begin{tikzcd}[column sep=12.0ex]
J X_\alpha
\rar[transform canvas={yshift=0.75ex}]{J s_{\alpha, \beta}} 
\rar[transform canvas={yshift=-0.75ex}, swap]{\iota_{X_\beta} \circ s_{\alpha+1, \beta} \circ q_\alpha} &
J X_\beta
\end{tikzcd}
\]
for all $\alpha < \beta$.
(In particular, $q_0 : J X_0 \to X_1$ is an isomorphism.) 
We define $s_{\beta+1, \beta+1} = \id_{X_{\beta+1}}$, $s_{\beta, \beta+1} = q_\beta \circ \iota_{X_\beta}$, and $s_{\alpha, \beta+1} = s_{\beta, \beta+1} \circ s_{\alpha, \beta}$ for all $\alpha < \beta$, so that we obtain a chain $X_{\bullet} : \parens{\beta + 2} \to \mathcal{C}$.

\item For each limit ordinal $\gamma \le \kappa$, given $X_\alpha$ for all $\alpha < \gamma$ and $s_{\alpha, \beta}$ for all $\alpha \le \beta < \gamma$, we define $X_\gamma = \indlim_{\alpha < \gamma} X_\alpha$ and $s_{\gamma,\gamma} = \id_{X_\gamma}$ and, for $\alpha < \gamma$, we define $s_{\alpha, \gamma}$ to be the components of the colimiting cocone.
\end{itemize}

Let $\bar{X} = X_\kappa$.
By construction, for all $\alpha \le \beta < \kappa$, the diagram in $\mathcal{C}$ shown below commutes,
\[
\begin{tikzcd}
J X_\alpha \dar[swap]{q_\alpha} \rar{J s_{\alpha, \beta}} &
J X_\beta \dar{q_\beta} \\
X_{\alpha+1} \rar[swap]{s_{\alpha+1, \beta+1}} &
X_{\beta+1}
\end{tikzcd}
\]
and by hypothesis, the morphisms $J s_{\alpha, \kappa} : J X_\alpha \to J X_\kappa$ constitute a colimiting cocone for the evident chain $J X_{\bullet} : \kappa \to \mathcal{C}$, so there is a unique morphism $\bar{q} : J \bar{X} \to \bar{X}$ such that $\bar{q} \circ J s_{\alpha, \kappa} = s_{\alpha+1, \kappa} \circ q_\alpha$ for all $\alpha < \kappa$.
Moreover,
\[
\parens{\bar{q} \circ \iota_{X_\kappa}} \circ s_{\alpha, \kappa} = \bar{q} \circ J s_{\alpha, \kappa} \circ \iota_{X_\alpha} = s_{\alpha+1, \kappa} \circ q_\alpha \circ \iota_{X_\alpha} = s_{\alpha+1, \kappa} \circ s_{\alpha, \alpha+1} = s_{\alpha, \kappa}
\]
so $\bar{q} \circ \iota_{\bar{X}} = \id_{\bar{X}}$, \ie $\tuple{\bar{X}, \bar{q}}$ is a $\tuple{J, \iota}$-algebra.

Define $\eta_X : X \to \bar{X}$ to be $s_{0, \kappa}$.
We will now show that $\tuple{\bar{X}, \bar{q}}$ is a \emph{free} $\tuple{J, \iota}$-algebra with unit $\eta_X$.
Let $\tuple{Y, r}$ be any $\tuple{J, \iota}$-algebra and let $f : X \to Y$ be any morphism in $\mathcal{C}$.
We construct a morphism $f_\alpha : X_\alpha \to Y$ for each ordinal $\alpha \le \kappa$ by transfinite recursion:
\begin{itemize}
\item We define $f_0 = f$.

\item For each ordinal $\beta < \kappa$, given $f_\alpha$ for all $\alpha \le \beta$ such that the following equations are satisfied,
\begin{align*}
f_\beta \circ s_{\alpha, \beta} & = f_\alpha && \text{for all } \alpha \le \beta \\
f_{\alpha+1} \circ q_\alpha & = r \circ J f_\alpha && \text{for all } \alpha < \beta
\end{align*}
we also have
\begin{align*}
\parens{r \circ J f_\beta} \circ \parens{\iota_{X_\beta} \circ s_{\alpha+1, \beta} \circ q_\alpha}
& = r \circ \iota_Y \circ f_\beta \circ s_{\alpha+1, \beta} \circ q_\alpha \\
& = f_\beta \circ s_{\alpha+1, \beta} \circ q_\alpha \\
& = f_{\alpha+1} \circ q_\alpha \\
& = r \circ J f_\alpha \\
& = \parens{r \circ J f_\beta} \circ J s_{\alpha, \beta} && \text{for all } \alpha < \beta
\end{align*}
so we may define $f_{\beta+1}$ to be the unique morphism $X_{\beta+1} \to Y$ in $\mathcal{C}$ such that $f_{\beta+1} \circ q_\beta = r \circ J f_\beta$.
Then,
\[
f_{\beta+1} \circ s_{\beta, \beta+1} = f_{\beta+1} \circ q_\beta \circ \iota_{X_\beta} = r \circ J f_\beta \circ \iota_{X_\beta} = r \circ \iota_Y \circ f_\beta = f_\beta
\]
so we have $f_{\beta+1} \circ s_{\alpha, \beta+1} = f_\alpha$ for all $\alpha \le \beta + 1$.

\item For each limit ordinal $\gamma \le \kappa$, we define $f_\gamma$ to be the unique morphism $X_\gamma \to Y$ in $\mathcal{C}$ such that $f_\gamma \circ s_{\alpha, \gamma} = f_\alpha$ for all $\alpha < \gamma$.

\end{itemize}

By construction, for all ordinals $\alpha < \kappa$,
\[
\parens{r \circ J f_\kappa} \circ J s_{\alpha, \kappa}
= r \circ J f_\alpha
= f_{\alpha+1} \circ q_\alpha
= f_\kappa \circ s_{\alpha+1, \kappa} \circ q_\alpha
= \parens{f_\kappa \circ \bar{q}} \circ J s_{\alpha, \kappa}
\]
so $r \circ J f_\kappa = f_\kappa \circ \bar{q}$, \ie $f_\kappa : X_\kappa \to Y$ is a $\tuple{J, \iota}$-algebra homomorphism $\tuple{X_\kappa, \bar{q}} \to \tuple{Y, r}$.
Moreover, for any homomorphism $\bar{f} : \tuple{X_\kappa, \bar{q}} \to \tuple{Y, r}$ and any ordinal $\alpha < \kappa$,
\[
\parens{\bar{f} \circ s_{\alpha+1, \kappa}} \circ q_\alpha = \bar{f} \circ \bar{q} \circ J s_{\alpha, \kappa} = r \circ J \bar{f} \circ J s_{\alpha, \kappa}
\]
so if $\bar{f} \circ s_{\alpha, \kappa} = f_\alpha$, then $\bar{f} \circ s_{\alpha+1, \kappa} = f_{\alpha+1}$; and for any limit ordinal $\gamma \le \kappa$, if $\bar{f} \circ s_{\alpha, \kappa} = f_\alpha$ for all $\alpha < \gamma$, then $\bar{f} \circ s_{\gamma, \kappa} = f_\gamma$ as well.
In particular, if $\bar{f} \circ \eta_X = f$, then $\bar{f} = f_\kappa$ by transfinite induction.
Thus, there is a unique homomorphism $\bar{f} : \tuple{X_\kappa, \bar{q}} \to \tuple{Y, r}$ such that $\bar{f} \circ \eta_X = f$.

The above argument shows that the comma category $\commacat{X}{U}$ has an initial object, and it is well known that $U$ has a left adjoint if and only if each comma category $\commacat{X}{U}$ has an initial object, so this completes the proof of (i).
For (ii), we simply observe that $\Kompakt[\lambda]{\mathcal{C}}$ is closed under colimits of $\lambda$-small diagrams in $\mathcal{C}$ (by \autoref{lem:small.objects}), so the above construction can be carried out entirely in $\Kompakt[\lambda]{\mathcal{C}}$.
\end{proof}

\begin{thm}[The category of algebras for a accessible pointed endofunctor]
\label{thm:accessible.pointed.endofunctors}
Let $J : \mathcal{C} \to \mathcal{C}$ be a functor, let $\iota : \id_\mathcal{C} \hoto J$ be a natural transformation, and let $\mathcal{C}^{\tuple{J, \iota}}$ be the category of algebras for the pointed endofunctor $\tuple{J, \iota}$.
\begin{enumerate}[(i)]
\item If $\mathcal{C}$ has colimits of small $\kappa$-filtered diagrams and $J : \mathcal{C} \to \mathcal{C}$ preserves them, then the forgetful functor $U : \mathcal{C}^{\tuple{J, \iota}} \to \mathcal{C}$ creates colimits of small $\kappa$-filtered diagrams; and if $\mathcal{C}$ is complete, then $U : \mathcal{C}^{\tuple{J, \iota}} \to \mathcal{C}$ also creates limits for all small diagrams.

\item If $\mathcal{C}$ is an accessible functor, then $\mathcal{C}^{\tuple{J, \iota}}$ is an accessible category.

\item If $\mathcal{C}$ has joint coequalisers for $\kappa$-small families of parallel pairs and colimits of chains of length $\le \kappa$ and $J : \mathcal{C} \to \mathcal{C}$ preserves colimits of $\kappa$-chains, then $U : \mathcal{C}^{\tuple{J, \iota}} \to \mathcal{C}$ is a monadic functor.
\end{enumerate}
\end{thm}
\begin{proof} \openproof
(i).
This is analogous to the well known fact about monads: \confer Propositions 4.3.1 and 4.3.2 in \citep{Borceux:1994b}.

\medskip\noindent
(ii).
We may construct $\mathcal{C}^{\tuple{J, \iota}}$ using inserters and equifiers, as in the proof of Theorem 2.78 in \citep{LPAC}.

\medskip\noindent
(iii).
The hypotheses of \autoref{thm:free.algebra.sequence} are satisfied, so the forgetful functor $U : \mathcal{C}^{\tuple{J, \iota}} \to \mathcal{C}$ has a left adjoint.
It is not hard to check that the other hypotheses of Beck's monadicity theorem are satisfied, so $U$ is indeed a monadic functor.
\end{proof}

\begin{thm}[The category of algebras for a strongly accessible pointed endofunctor]
\label{thm:strongly.accessible.pointed.endofunctors}
Let $\mathcal{C}$ be a locally $\lambda$-presentable category, let $J : \mathcal{C} \to \mathcal{C}$ be a functor that preserves colimits of small $\kappa$-filtered diagrams, let $\iota : \id_\mathcal{C} \hoto J$ be a natural transformation, and let $\mathbb{T} = \tuple{T, \eta, \mu}$ be the induced monad on $\mathcal{C}$.
If $J : \mathcal{C} \to \mathcal{C}$ is a strongly $\lambda$-accessible functor and $\kappa < \lambda$, then:
\begin{enumerate}[(i)]
\item The functor $T : \mathcal{C} \to \mathcal{C}$ preserves colimits of small $\kappa$-filtered diagrams and is strongly $\lambda$-accessible.

\item $\mathcal{C}^{\tuple{J, \iota}}$ is a locally $\lambda$-presentable category.

\item The forgetful functor $U : \mathcal{C}^{\tuple{J, \iota}} \to \mathcal{C}$ is a strongly $\lambda$-accessible functor.
\end{enumerate}
\end{thm}
\begin{proof}
(i).
By \autoref{thm:accessible.pointed.endofunctors}, the forgetful functor $U : \mathcal{C}^{\tuple{J, \iota}} \to \mathcal{C}$ creates colimits of small $\kappa$-filtered diagrams when $J : \mathcal{C} \to \mathcal{C}$ preserves colimits of small $\kappa$-filtered diagrams, so $T : \mathcal{C} \to \mathcal{C}$ must also preserve these colimits.
Moreover, \autoref{thm:free.algebra.sequence} implies $T : \mathcal{C} \to \mathcal{C}$ is strongly $\lambda$-accessible if $J : \mathcal{C} \to \mathcal{C}$ is.

\medskip\noindent
(ii).
It is not hard to check that the forgetful functor $\mathcal{C}^{\tuple{J, \iota}} \to \mathcal{C}$ is a monadic functor, so the claim reduces to the fact that $\mathcal{C}^\mathbb{T}$ is a locally $\lambda$-presentable category if $T : \mathcal{C} \to \mathcal{C}$ is a $\lambda$-accessible functor.
\unskip\footnote{See Theorem 2.78 and the following remark in \citep{LPAC}, or Theorem 5.5.9 in \citep{Borceux:1994b}.}

\medskip\noindent
(iii).
Apply \autoref{thm:strongly.accessible.monads}.
\end{proof}
\section{Accessibly generated categories}
\label{sect:accessibly.generated.categories}

\begin{numpar}
Throughout this section, $\kappa$ and $\lambda$ are regular cardinals such that $\kappa \le \lambda$.
\end{numpar}

\begin{dfn}
A \strong{$\tuple{\kappa, \lambda}$-accessibly generated category} is an essentially small category $\mathcal{C}$ that satisfies the following conditions:
\begin{itemize}
\item Every $\lambda$-small $\kappa$-filtered diagram in $\mathcal{C}$ has a colimit in $\mathcal{C}$.

\item Every object in $\mathcal{C}$ is (the object part of) a colimit of some $\lambda$-small $\kappa$-filtered diagram of $\tuple{\kappa, \lambda}$-presentable objects in $\mathcal{C}$.
\end{itemize}
\end{dfn}

\begin{remark*}
In the case where $\lambda$ is a strongly inaccessible cardinal with $\kappa < \lambda$, the concept of $\tuple{\kappa, \lambda}$-accessibly generated categories is very closely related to the concept of class-$\kappa$-accessible categories (in the sense of \citet{Chorny-Rosicky:2012}) relative to the universe of hereditarily $\lambda$-small sets, though there are some technical differences.
For our purposes, we do not need to assume that $\lambda$ is a strongly inaccessible cardinal.
\end{remark*}

\begin{remark}
\label{rem:idempotent-complete.categories.are.accessibly.generated}
\Autoref{lem:very.small.filtered.categories} says that every $\kappa$-small $\kappa$-filtered category has a cofinal idempotent, so every object is automatically $\tuple{\kappa, \kappa}$-presentable.
Thus, an essentially small category is $\tuple{\kappa, \kappa}$-accessibly generated if and only if it is idempotent-complete, \ie if and only if all idempotent endomorphisms in $\mathcal{C}$ split.
\end{remark}

\begin{remark}
In the definition of `$\tuple{\kappa, \lambda}$-accessibly generated category', we can replace `essentially small category' with `locally small category such that the full subcategory of $\tuple{\kappa, \lambda}$-presentable objects is essentially small'.

\end{remark}

\begin{prop}
\label{prop:presentable.objects.in.categories.of.presentable.objects}
Let $\mathcal{C}$ be a $\kappa$-accessible category.
\begin{enumerate}[(i)]
\item $\Kompakt[\kappa]{\mathcal{C}}$ is a $\tuple{\kappa, \kappa}$-accessibly generated category, and every object in $\Kompakt[\kappa]{\mathcal{C}}$ is $\tuple{\kappa, \kappa}$-presentable.

\item If $\kappa \sharplylt \lambda$, then $\Kompakt[\lambda]{\mathcal{C}}$ is a $\tuple{\kappa, \lambda}$-accessibly generated category, and the $\tuple{\kappa, \lambda}$-presentable objects in $\Kompakt[\lambda]{\mathcal{C}}$ are precisely the $\kappa$-presentable objects in $\mathcal{C}$.
\end{enumerate}
\end{prop}
\begin{proof}
Combine \autoref{lem:small.objects}, \autoref{prop:presentable.objects.are.presentable}, and \autoref{rem:idempotent-complete.categories.are.accessibly.generated}.
\end{proof}

\begin{dfn}
Let $\mu$ be a regular cardinal such that $\lambda \le \mu$.
A \strong{$\tuple{\kappa, \lambda, \mu}$-accessibly generated extension} is a functor $F : \mathcal{A} \to \mathcal{B}$ with the following properties:
\begin{itemize}
\item $\mathcal{A}$ is a $\tuple{\kappa, \lambda}$-accessibly generated category.

\item $\mathcal{B}$ is a $\tuple{\kappa, \mu}$-accessibly generated category.

\item $F : \mathcal{A} \to \mathcal{B}$ preserves colimits of $\lambda$-small $\kappa$-filtered diagrams.

\item $F$ sends $\tuple{\kappa, \lambda}$-presentable objects in $\mathcal{A}$ to $\tuple{\kappa, \mu}$-presentable objects in $\mathcal{B}$.

\item The induced functor $F : \Kompakt[\kappa][\lambda]{\mathcal{A}} \to \Kompakt[\kappa][\mu]{\mathcal{B}}$ is fully faithful and essentially surjective on objects.
\end{itemize}
\end{dfn}

\begin{remark*}
The concept of accessibly generated extensions is essentially a generalisation of the concept of accessible extensions, as defined in \citep{Low:2013a}.
\end{remark*}

\begin{remark}
\label{rem:inclusion.of.generators.is.a.accessibly.generated.extension}
Let $\mathcal{C}$ be a $\tuple{\kappa, \lambda}$-accessibly generated category.
Then, in view of \autoref{rem:idempotent-complete.categories.are.accessibly.generated}, the inclusion $\Kompakt[\kappa][\lambda]{\mathcal{C}} \embedinto \mathcal{C}$ is a $\tuple{\kappa, \kappa, \lambda}$-accessibly generated extension.
\end{remark}

\begin{lem}
\label{lem:composition.of.accessibly.generated.extensions}
Let $F : \mathcal{A} \to \mathcal{B}$ be a $\tuple{\kappa, \lambda, \mu}$-accessibly generated extension and let $G : \mathcal{B} \to \mathcal{C}$ be a $\tuple{\kappa, \mu, \nu}$-accessibly generated extension.
If $\lambda \le \mu$, then the composite $G F : \mathcal{A} \to \mathcal{C}$ is a $\tuple{\kappa, \lambda, \nu}$-accessibly generated extension.
\end{lem}
\begin{proof} \obviousproof
Straightforward.
\end{proof}

\begin{lem}
\label{lem:accessible.comparison.lemma}
Let $F : \mathcal{A} \to \mathcal{B}$ be a $\tuple{\kappa, \kappa, \lambda}$-accessibly generated extension.
\begin{enumerate}[(i)]
\item There is a functor $U : \mathcal{B} \to \Ind[\kappa]{\mathcal{A}}$ equipped with a natural bijection of the form below,
\[
\Hom[{\Ind[\kappa]{\mathcal{A}}}]{A}{U B} \cong \Hom[\mathcal{B}]{F A}{B}
\]
and it is unique up to unique isomorphism.

\item Moreover, the functor $U : \mathcal{B} \to \Ind[\kappa]{\mathcal{A}}$ is fully faithful and preserves colimits of $\lambda$-small $\kappa$-filtered diagrams.

\item In particular, $F : \mathcal{A} \to \mathcal{B}$ is a fully faithful functor.

\item If $\kappa \sharplylt \lambda$, then the $\lambda$-accessible functor $\bar{U} : \Ind[\lambda]{\mathcal{B}} \to \Ind[\kappa]{\mathcal{A}}$ induced by $U : \mathcal{B} \to \Ind[\kappa]{\mathcal{A}}$ is fully faithful and essentially surjective on objects.

\item In particular, if $\kappa \sharplylt \lambda$, then $\Ind[\lambda]{\mathcal{B}}$ is a $\kappa$-accessible category.
\end{enumerate}
\end{lem}
\begin{proof}
(i).
Let $B$ be an object in $\mathcal{B}$.
By hypothesis, there is a $\lambda$-small $\kappa$-filtered diagram $X : \mathcal{J} \to \mathcal{A}$ such that $B \cong {\indlim}_\mathcal{J} F X$.
Then, for every object $A$ in $\mathcal{A}$,
\[
\Hom[\mathcal{B}]{F A}{B} \cong \indlim_\mathcal{J} \Hom[\mathcal{B}]{F A}{F X} \cong {\indlim}_\mathcal{J} \Hom[\mathcal{A}]{A}{X}
\]
so there is an object $U B$ in $\Ind[\kappa]{\mathcal{A}}$ such that 
\[
\Hom[{\Ind[\kappa]{\mathcal{A}}}]{A}{U B} \cong \Hom[\mathcal{B}]{F A}{B}
\]
for all objects $A$ in $\mathcal{A}$, and an object with such a natural bijection is unique up to unique isomorphism, by \autoref{thm:universal.property.of.free.ind.completion} and \autoref{prop:dense.generators.for.accessible.categories}.
A similar argument can be used to define $U g$ for morphisms $g : B_0 \to B_1$ in $\mathcal{B}$, and it is straightforward to check that this indeed defines a functor $U : \mathcal{B} \to \Ind[\kappa]{\mathcal{A}}$.

\medskip\noindent
(ii).
Let $Y : \mathcal{J} \to \mathcal{B}$ be a $\lambda$-small $\kappa$-filtered diagram in $\mathcal{B}$.
Then, for any object $A$ in $\mathcal{A}$,
\begin{align*}
\Hom[\mathcal{B}]{F A}{\indlim_\mathcal{J} Y} 
& \cong \indlim_\mathcal{J} \Hom[\mathcal{B}]{F A}{Y} \\
& \cong \indlim_\mathcal{J} \Hom[{\Ind[\kappa]{\mathcal{A}}}]{A}{U Y} \\
& \cong \Hom[{\Ind[\kappa]{\mathcal{A}}}]{A}{\indlim_\mathcal{J} U Y}
\end{align*}
so $U : \mathcal{B} \to \Ind[\kappa]{\mathcal{A}}$ indeed preserves colimits of $\lambda$-small $\kappa$-filtered diagrams.
A similar argument can be used to show that $U : \mathcal{B} \to \Ind[\kappa]{\mathcal{A}}$ is fully faithful.

\medskip\noindent
(iii).
The composite $U F : \mathcal{A} \to \Ind[\kappa]{\mathcal{A}}$ is clearly fully faithful, so it follows from (ii) that $F : \mathcal{A} \to \mathcal{B}$ is fully faithful.

\medskip\noindent
(iv).
\Autoref{prop:presentable.objects.are.presentable} implies that $U : \mathcal{B} \to \Ind[\kappa]{\mathcal{A}}$ is essentially surjective onto the full subcategory of $\lambda$-presentable objects in $\Ind[\kappa]{\mathcal{A}}$.
Moreover, since $\kappa \sharplylt \lambda$, $\Ind[\kappa]{\mathcal{A}}$ is also a $\lambda$-accessible category,
\unskip\footnote{See Theorem~2.3.10 in \citep{Makkai-Pare:1989} or Theorem~2.11 in \citep{LPAC}.}
and it follows that the induced $\lambda$-accessible functor $\Ind[\lambda]{\mathcal{B}} \to \Ind[\kappa]{\mathcal{A}}$ is fully faithful and essentially surjective on objects.

\medskip\noindent
(v).
We know that $\Ind[\kappa]{\mathcal{A}}$ is a $\kappa$-accessible category, so it follows from (iv) that $\Ind[\lambda]{\mathcal{B}}$ is also a $\kappa$-accessible category.
\end{proof}

\begin{prop}
\label{prop:accessibly.generated.extensions}
Let $F : \mathcal{A} \to \mathcal{B}$ be a $\tuple{\kappa, \lambda, \mu}$-accessibly generated extension.
Assuming either $\kappa = \lambda$ or $\kappa \sharplylt \lambda$:
\begin{enumerate}[(i)]
\item There is a functor $U : \mathcal{B} \to \Ind[\lambda]{\mathcal{A}}$ equipped with a natural bijection of the form below,
\[
\Hom[{\Ind[\lambda]{\mathcal{A}}}]{A}{U B} \cong \Hom[\mathcal{B}]{F A}{B}
\]
and it is unique up to unique isomorphism.

\item Moreover, the functor $U : \mathcal{B} \to \Ind[\lambda]{\mathcal{A}}$ is fully faithful and preserves colimits of $\mu$-small $\lambda$-filtered diagrams.

\item In particular, $F : \mathcal{A} \to \mathcal{B}$ is a fully faithful functor.

\item If $\lambda \sharplylt \mu$, then the $\mu$-accessible functor $\bar{U} : \Ind[\mu]{\mathcal{B}} \to \Ind[\lambda]{\mathcal{A}}$ induced by $U : \mathcal{B} \to \Ind[\lambda]{\mathcal{A}}$ is fully faithful and essentially surjective on objects.

\item In particular, if $\lambda \sharplylt \mu$, then $\Ind[\mu]{\mathcal{B}}$ is a $\kappa$-accessible category.
\end{enumerate}
\end{prop}
\begin{proof}
\Autoref{rem:inclusion.of.generators.is.a.accessibly.generated.extension} says the inclusion $\Kompakt[\kappa][\lambda]{\mathcal{A}} \embedinto \mathcal{A}$ is a $\tuple{\kappa, \kappa, \lambda}$-accessibly generated extension, so by \autoref{lem:composition.of.accessibly.generated.extensions}, the composite $\Kompakt[\kappa][\lambda]{\mathcal{A}} \embedinto \mathcal{A} \to \mathcal{B}$ is a $\tuple{\kappa, \kappa, \mu}$-accessible generated extension.
Moreover, $\kappa \sharplylt \mu$,
\unskip\footnote{See Proposition 2.3.2 in \citep{Makkai-Pare:1989}.}
so the claims follow, by (two applications of) \autoref{lem:accessible.comparison.lemma}.
\end{proof}

\begin{thm}
\label{thm:classification.of.accessibly.generated.categories}
If either $\kappa = \lambda$ or $\kappa \sharplylt \lambda$, then the following are equivalent for a idempotent-complete category $\mathcal{C}$:
\begin{enumerate}[(i)]
\item $\mathcal{C}$ is a $\tuple{\kappa, \lambda}$-accessibly generated category.

\item $\Ind[\lambda]{\mathcal{C}}$ is a $\kappa$-accessible category.

\item $\mathcal{C}$ is equivalent to $\Kompakt[\lambda]{\mathcal{D}}$ for some $\kappa$-accessible category $\mathcal{D}$.
\end{enumerate}
\end{thm}
\begin{proof}
(i) \implies (ii).
Apply \autoref{lem:accessible.comparison.lemma} to \autoref{rem:inclusion.of.generators.is.a.accessibly.generated.extension}.

\medskip\noindent
(ii) \implies (iii).
It is not hard to check that every $\lambda$-presentable object in $\Ind[\lambda]{\mathcal{C}}$ is a retract of some object in the image of the canonical embedding $\mathcal{C} \to \Ind[\lambda]{\mathcal{C}}$.
But $\mathcal{C}$ is idempotent-complete, so the canonical embedding is fully faithful and essentially surjective onto the full subcategory of $\lambda$-presentable objects in $\Ind[\lambda]{\mathcal{C}}$.

\medskip\noindent
(iii) \implies (i).
See \autoref{prop:presentable.objects.in.categories.of.presentable.objects}.
\end{proof}

\begin{cor}
If $\mathcal{C}$ is a $\tuple{\kappa, \lambda}$-accessibly generated category, then so is $\Func{\mathbf{2}}{\mathcal{C}}$.
\end{cor}
\begin{proof}
Combine \autoref{cor:accessible.arrow.categories} and \autoref{thm:classification.of.accessibly.generated.categories}.
\end{proof}
\section{Accessible factorisation systems}
\label{sect:accessible.factorisation.systems}

\begin{numpar}
Throughout this section, $\kappa$ is an arbitrary regular cardinal.
\end{numpar}

\begin{lem}
\label{lem:filtered.colimits.in.the.right.class.of.cofibrantly-generated.wfs}
Let $\mathcal{C}$ be a category with colimits of small $\kappa$-filtered diagrams, let $\mathcal{I}$ be a subset of $\mor \mathcal{C}$, and let $\rlpclass{\mathcal{I}}$ be the class of morphisms in $\mathcal{C}$ with the right lifting property with respect to $\mathcal{I}$.
If the domains and codomains of the members of $\mathcal{I}$ are $\kappa$-presentable objects in $\mathcal{C}$, then $\rlpclass{\mathcal{I}}$ (regarded as a full subcategory of $\Func{\mathbf{2}}{\mathcal{C}}$) is closed under colimits of small $\kappa$-filtered diagrams in $\Func{\mathbf{2}}{\mathcal{C}}$.
\end{lem}
\begin{proof}
By \autoref{prop:presentable.objects.in.diagram.categories}, any element of $\mathcal{I}$ is $\kappa$-presentable as an object in $\Func{\mathbf{2}}{\mathcal{C}}$.
Thus, given any morphism $\phi : e \to \indlim_\mathcal{J} f$ in $\Func{\mathbf{2}}{\mathcal{C}}$ where $e$ is in $\mathcal{I}$ and $f : \mathcal{J} \to \Func{\mathbf{2}}{\mathcal{C}}$ is a small $\kappa$-filtered diagram with each vertex in $\rlpclass{\mathcal{I}}$, $\phi$ must factor through $f j \to \indlim_\mathcal{J} f$ for some $j$ in $\mathcal{J}$ (by considering $\indlim_\mathcal{J} \Hom[\Func{\mathbf{2}}{\mathcal{C}}]{e}{f}$) and so we can construct the required lift.
\end{proof}

\begin{lem}
\label{lem:solution.set.condition.and.lifting.conditions}
Let $\mathcal{C}$ be a $\kappa$-accessible category and let $\mathcal{R}$ be a $\kappa$-accessible full subcategory of $\Func{\mathbf{2}}{\mathcal{C}}$.
If $g : Z \to W$ is a morphism in $\mathcal{C}$ where both $Z$ and $W$ are $\kappa$-presentable objects in $\mathcal{C}$, then:
\begin{enumerate}[(i)]
\item Given a morphism $f : X \to Y$ in $\mathcal{C}$ that is in $\mathcal{R}$, any morphism $g \to f$ in $\Func{\mathbf{2}}{\mathcal{C}}$ admits a factorisation of the form $g \to f' \to f$ where $f'$ is in $\Kompakt[\kappa]{\mathcal{R}}$.

\item The morphism $g : Z \to W$ has the left lifting property with respect to $\mathcal{R}$ if and only if it has the left lifting property with respect to $\Kompakt[\kappa]{\mathcal{R}}$.
\end{enumerate}
\end{lem}
\begin{proof}
(i).
\Autoref{prop:presentable.objects.in.diagram.categories} says that $g$ is a $\kappa$-presentable object in $\Func{\mathbf{2}}{\mathcal{C}}$; but every object in $\mathcal{R}$ is the colimit of a small $\kappa$-filtered diagram of $\kappa$-presentable objects in $\mathcal{R}$, and the inclusion $\mathcal{R} \embedinto \Func{\mathbf{2}}{\mathcal{C}}$ is $\kappa$-accessible, so any morphism $g \to f$ must factor through some $\kappa$-presentable object in $\mathcal{R}$.

\medskip\noindent
(ii).
If $g$ has the left lifting property with respect to $\mathcal{R}$, then it certainly has the left lifting property with respect to $\Kompakt[\kappa]{\mathcal{R}}$.
Conversely, by factorising morphisms $g \to f$ as in (i), we see that $g$ has the left lifting property with respect to $\mathcal{R}$ as soon as it has the left lifting property with respect to $\Kompakt[\kappa]{\mathcal{R}}$.
\end{proof}

\begin{thm}[Quillen's small object argument]
\label{thm:Quillen.small.object.argument}
Let $\kappa$ be a regular cardinal, let $\mathcal{C}$ be a locally $\kappa$-presentable category, and let $\mathcal{I}$ be a small subset of $\mor \mathcal{C}$.
\begin{enumerate}[(i)]
\item There exists a functorial weak factorisation system $\tuple{L, R}$ on $\mathcal{C}$ whose right class is $\rlpclass{\mathcal{I}}$; in particular, there is a weak factorisation system on $\mathcal{C}$ cofibrantly generated by $\mathcal{I}$.

\item If the morphisms that are in $\mathcal{I}$ are $\kappa$-presentable as objects in $\Func{\mathbf{2}}{\mathcal{C}}$, then $\tuple{L, R}$ can be chosen so that the functors $L, R : \Func{\mathbf{2}}{\mathcal{C}} \to \Func{\mathbf{2}}{\mathcal{C}}$ are $\kappa$-accessible.

\item In addition, if $\lambda$ is a regular cardinal such that every hom-set of $\Kompakt[\kappa]{\mathcal{C}}$ is $\lambda$-small, $\mathcal{I}$ is $\lambda$-small, and $\kappa \sharplylt \lambda$, then $\tuple{L, R}$ can be chosen so that the functors $L, R : \Func{\mathbf{2}}{\mathcal{C}} \to \Func{\mathbf{2}}{\mathcal{C}}$ preserve $\lambda$-presentable objects.
\end{enumerate}
\end{thm}
\begin{proof} \openproof
(i). See \eg Proposition 10.5.16 in \citep{Hirschhorn:2003}.

\medskip\noindent
(ii) and (iii). These claims can be verified by tracing the construction of $L$ and $R$ and applying lemmas \ref{lem:small.objects} and \ref{lem:estimate.of.morphisms.in.accessible.categories}.
\end{proof}

\begin{remark}
The algebraically free natural weak factorisation system produced by Garner's small object argument \citep{Garner:2009} satisfy claims (ii) and (iii) of the above theorem (under the same hypotheses). The proof is somewhat more straightforward, because the right half of the resulting algebraic factorisation system can be described in terms of a certain density comonad.
\end{remark}

\begin{prop}
\label{prop:accessibility.of.right.class.of.accessible.fwfs}
Let $\mathcal{C}$ be a locally presentable category, let $\tuple{L, R}$ be a functorial weak factorisation system on $\mathcal{C}$, and let $\lambda : \id_{\Func{\mathbf{2}}{\mathcal{C}}} \hoto R$ be the natural transformation whose component at an object $f$ in $\Func{\mathbf{2}}{\mathcal{C}}$ corresponds to the following commutative square in $\mathcal{C}$:
\[
\begin{tikzcd}
\bullet \dar[swap]{f} \rar{L f} &
\bullet \dar{R f} \\
\bullet \rar[equals] &
\bullet
\end{tikzcd}
\]
Let $\mathcal{R}$ be the full subcategory of $\Func{\mathbf{2}}{\mathcal{C}}$ spanned by the morphisms in $\mathcal{C}$ that are in the right class of the induced weak factorisation system.
\begin{enumerate}[(i)]
\item $\mathcal{R}$ is also the full subcategory of $\Func{\mathbf{2}}{\mathcal{C}}$ spanned by the image of the forgetful functor $\Func{\mathbf{2}}{\mathcal{C}}^{\tuple{R, \lambda}} \to \Func{\mathbf{2}}{\mathcal{C}}$, where $\Func{\mathbf{2}}{\mathcal{C}}^{\tuple{R, \lambda}}$ is the category of algebras for the pointed endofunctor $\tuple{R, \lambda}$.

\item If $R : \Func{\mathbf{2}}{\mathcal{C}} \to \Func{\mathbf{2}}{\mathcal{C}}$ is an accessible functor, then $\Func{\mathbf{2}}{\mathcal{C}}^{\tuple{R, \lambda}}$ is a locally presentable category, and the forgetful functor $\Func{\mathbf{2}}{\mathcal{C}}^{\tuple{R, \lambda}} \to \Func{\mathbf{2}}{\mathcal{C}}$ is monadic.

\item If $R : \Func{\mathbf{2}}{\mathcal{C}} \to \Func{\mathbf{2}}{\mathcal{C}}$ is strongly $\pi$-accessible and preserves colimits of $\kappa$-filtered diagrams, where $\kappa < \pi$, and $\mathcal{R}$ is closed under colimits of small $\pi$-filtered diagrams in $\Func{\mathbf{2}}{\mathcal{C}}$, then $\mathcal{R}$ is a $\pi$-accessible subcategory of $\Func{\mathbf{2}}{\mathcal{C}}$.
\end{enumerate}
\end{prop}
\begin{proof}
(i).
This is a special case of \autoref{prop:algebras.and.coalgebras.for.fwfs}.

\medskip\noindent
(ii).
Apply \autoref{thm:accessible.pointed.endofunctors}.

\medskip\noindent
(iii).
By \autoref{thm:strongly.accessible.pointed.endofunctors}, $\Func{\mathbf{2}}{\mathcal{C}}^{\tuple{R, \lambda}}$ is a locally $\pi$-presentable category, and the forgetful functor $\Func{\mathbf{2}}{\mathcal{C}}^{\tuple{R, \lambda}} \to \Func{\mathbf{2}}{\mathcal{C}}$ is moreover strongly $\pi$-accessible.
Thus, we may apply \autoref{prop:image.of.strongly.accessible.functors} to (i) and deduce that $\mathcal{R}$ is a $\pi$-accessible subcategory.
\end{proof}

\begin{prop}
\label{prop:accessibility.of.right.class.of.cofibrantly-generated.wfs}
Let $\mathcal{C}$ be a locally presentable category, and let $\mathcal{I}$ be a subset of $\mor \mathcal{C}$.
Then $\rlpclass{\mathcal{I}}$, considered as a full subcategory of $\Func{\mathbf{2}}{\mathcal{C}}$, is an accessible subcategory.
\end{prop}
\begin{proof}
Combine \autoref{thm:Quillen.small.object.argument} and \autoref{prop:accessibility.of.right.class.of.accessible.fwfs}.
\end{proof}

\section{Strongly combinatorial model categories}
\label{sect:strongly.combinatorial.model.categories}

To apply the results of the previous section to the theory of combinatorial model categories, it is useful to collect some convenient hypotheses together as a definition:

\begin{dfn}
Let $\kappa$ and $\lambda$ be regular cardinals.
A \strong{strongly $\tuple{\kappa, \lambda}$-\hspace{0pt}combinatorial model category} is a combinatorial model category $\mathcal{M}$ that satisfies these axioms:
\begin{itemize}
\item $\mathcal{M}$ is a locally $\kappa$-presentable category, and $\kappa \sharplylt \lambda$.

\item $\Kompakt[\lambda]{\mathcal{M}}$ is closed under finite limits in $\mathcal{M}$.

\item Each hom-set in $\Kompakt[\kappa]{\mathcal{M}}$ is $\lambda$-small.

\item There exist $\lambda$-small sets of morphisms in $\Kompakt[\kappa]{\mathcal{M}}$ that cofibrantly generate the model structure of $\mathcal{M}$.
\end{itemize}
\end{dfn}

\begin{remark}
Let $\mathcal{M}$ be a strongly $\tuple{\kappa, \lambda}$-combinatorial model category and let $\lambda \sharplylt \mu$.
Then $\kappa \sharplylt \mu$, so by \autoref{lem:accessible.functors.are.strongly.accessible}, $\Kompakt[\mu]{\mathcal{M}}$ is also closed under finite limits.
Hence, $\mathcal{M}$ is also a strongly $\tuple{\kappa, \mu}$-combinatorial model category.
\end{remark}

\begin{example}
Let $\cat{\SSet}$ be the category of simplicial sets.
$\cat{\SSet}$, equipped with the  Kan--Quillen model structure, is a strongly $\tuple{\aleph_0, \aleph_1}$-combinatorial model category.
\end{example}

\begin{example}
Let $R$ be a ring, let $\cat{\Chcx{R}}$ be the category of unbounded chain complexes of left $R$-modules, and let $\lambda$ be an uncountable regular cardinal such that $R$ is $\lambda$-small (as a set).
\begin{itemize}
\item It is not hard to verify that $\cat{\Chcx{R}}$ is a locally $\aleph_0$-presentable category where the $\aleph_0$-presentable objects are the bounded chain complexes of finitely presented left $R$-modules.

\item The $\lambda$-presentable objects are precisely the chain complexes $M_{\bullet}$ such that $\sum_{n \in \mathbb{Z}} \card{M_n} < \lambda$, so the full subcategory of $\lambda$-presentable objects is closed under finite limits.

\item By considering matrices over $R$, we may deduce that the set of chain maps between any two $\aleph_0$-presentable objects in $\cat{\Chcx{R}}$ is $\lambda$-small.

\item The cofibrations in the projective model structure on $\cat{\Chcx{R}}$ are generated by a countable set of chain maps between $\aleph_0$-presentable chain complexes, as are the trivial cofibrations.

\end{itemize}
Thus, $\cat{\Chcx{R}}$ is a strongly $\tuple{\aleph_0, \lambda}$-combinatorial model category.
\end{example}

\begin{example}
Let $\cat{\SymSp}$ be the category of symmetric spectra of \citet{HSS:2000} and let $\lambda$ be a regular cardinal such that $\aleph_1 \sharplylt \lambda$ and $2^{\aleph_0} < \lambda$.
(Such a cardinal exists: for instance, we may take $\lambda$ to be the cardinal successor of $2^{2^{\aleph_0}}$; or, assuming the continuum hypothesis, we may take $\lambda = \aleph_2$.)
\begin{itemize}
\item The category of pointed simplicial sets, $\cat{\SSet_*}$, is locally $\aleph_0$-presentable; hence, so is the category $\Func{\mbfSigma}{\cat{\SSet_*}}$ of symmetric sequences of pointed simplicial sets, by \autoref{prop:locally.presentable.functor.categories}.
There is a symmetric monoidal closed structure on $\Func{\mbfSigma}{\cat{\SSet_*}}$ such that $\cat{\SymSp}$ is equivalent to the category of $S$-modules, where $S$ is (the underlying symmetric sequence of) the symmetric sphere spectrum defined in Example 1.2.4 in \opcit; thus, $\cat{\SymSp}$ is the category of algebras for an $\aleph_0$-accessible monad, hence is itself is a locally $\aleph_0$-presentable category.

\item Since (the underlying symmetric sequence of) $S$ is an $\aleph_1$-presentable object in $\Func{\mbfSigma}{\cat{\SSet}_*}$, we can apply \autoref{prop:presentable.objects.in.diagram.categories} and  \autoref{thm:strongly.accessible.monads} to deduce that the $\aleph_1$-presentable objects in $\cat{\SymSp}$ are precisely the ones whose underlying symmetric sequence consists of countable simplicial sets.
Hence, $\Kompakt[\aleph_1]{\cat{\SymSp}}$ is closed under finite limits, and the same is true for $\Kompakt[\lambda]{\cat{\SymSp}}$ because $\aleph_1 \sharplylt \lambda$.

\item It is clear that there are $\le 2^{\aleph_0}$ morphisms between two $\aleph_1$-presentable symmetric sequences; in particular, there are $< \lambda$ morphisms between two $\aleph_1$-presentable symmetric spectra.

\item The functor $\pblank_n : \cat{\SymSp} \to \cat{\SSet}$ that sends a symmetric spectrum $X$ to the simplicial set $X_n$ preserves filtered colimits, so by \autoref{prop:left.adjoints.and.strong.accessibility}, its left adjoint $F_n : \cat{\SSet} \to \cat{\SymSp}$ preserves $\aleph_0$-presentability.
Thus, the set of generating cofibrations for the stable model structure on $\cat{\SymSp}$ given by Proposition~3.4.2 in \opcit is a countable set of morphisms between $\aleph_0$-presentable symmetric spectra.

Using the fact that the mapping cylinder of a morphism between two $\aleph_1$-presentable symmetric spectra is also an $\aleph_1$-presentable symmetric spectrum, we deduce that the set of generating trivial cofibrations given in Definition~3.4.9 in \opcit is a countable set of morphisms between $\aleph_1$-presentable symmetric spectra.
\end{itemize}
We therefore conclude that $\cat{\SymSp}$ is a strongly $\tuple{\aleph_1, \lambda}$-combinatorial model category.
\end{example}

\begin{prop}
\label{prop:combinatorial.model.categories.are.strongly.combinatorial}
For any combinatorial model category $\mathcal{M}$, there exist regular cardinals $\kappa$ and $\lambda$ such that $\mathcal{M}$ is a strongly $\tuple{\kappa, \lambda}$-combinatorial model category.
\end{prop}
\begin{proof}
In view of \autoref{lem:accessible.functors.are.strongly.accessible}, this reduces to the fact that there are arbitrarily large $\lambda$ such that $\kappa \sharplylt \lambda$.
\unskip\footnote{See Corollary~2.3.6 in \citep{Makkai-Pare:1989}, or Example 2.13\hairkern (6) in \citep{LPAC}, or Corollary~5.4.8 in \citep{Borceux:1994b}.}
\end{proof}

\begin{prop}
\label{prop:accessibility.of.fibrations.and.trivial.fibrations}
Let $\mathcal{M}$ be a strongly $\tuple{\kappa, \lambda}$-combinatorial model category.
\begin{enumerate}[(i)]
\item There exist (trivial cofibration, fibration)- and (cofibration, trivial fi\-bra\-tion)-\allowhyphens factorisation functors that are $\kappa$-accessible and strongly $\lambda$-accessible.

\item Let $\mathcal{F}$ (\resp $\mathcal{F}'$) be the full subcategory of $\Func{\mathbf{2}}{\mathcal{M}}$ spanned by the fibrations (\resp trivial fibrations).
Then $\mathcal{F}$ and $\mathcal{F}'$ are closed under colimits of small $\kappa$-filtered diagrams in $\Func{\mathbf{2}}{\mathcal{M}}$.
\end{enumerate}
\end{prop}
\begin{proof}
(i).
Since the weak factorisation systems on $\mathcal{M}$ are cofibrantly generated by $\lambda$-small sets of morphisms in $\Kompakt[\kappa]{\mathcal{M}}$ and the hom-sets of $\Kompakt[\kappa]{\mathcal{M}}$ are all $\lambda$-small, we may apply  \autoref{thm:Quillen.small.object.argument} to obtain the required functorial weak factorisation systems.

\medskip\noindent
(ii).
This is a special case of \autoref{lem:filtered.colimits.in.the.right.class.of.cofibrantly-generated.wfs}.
\end{proof}

\begin{lem}
\label{lem:model.structure.criteria}
Let $\mathcal{M}$ be a category with limits and colimits of finite diagrams and let $\tuple{\mathcal{C}', \mathcal{F}}$ and $\tuple{\mathcal{C}, \mathcal{F}'}$ be weak factorisation systems on $\mathcal{M}$.
Assume $\mathcal{W}$ is a class of morphisms in $\mathcal{C}$ with the following property:
\[
\mathcal{W} \subseteq \set{ q \circ j }{ j \in \mathcal{C}', q \in \mathcal{F}' }
\]
The following are equivalent:
\begin{enumerate}[(i)]
\item $\tuple{\mathcal{C}, \mathcal{W}, \mathcal{F}}$ is a model structure on $\mathcal{M}$.

\item $\mathcal{W}$ has the 2-out-of-3 property in $\mathcal{M}$, $\mathcal{C}' = \mathcal{C} \cap \mathcal{W}$, and $\mathcal{F}' = \mathcal{W} \cap \mathcal{F}$.

\item $\mathcal{W}$ has the 2-out-of-3 property in $\mathcal{M}$, $\mathcal{C}' \subseteq \mathcal{W}$, and $\mathcal{F}' = \mathcal{W} \cap \mathcal{F}$.
\end{enumerate}
\end{lem}
\begin{proof}
(i) \implies (ii).
Use the retract argument.

\medskip\noindent
(ii) \implies (iii).
Immediate.

\medskip\noindent
(iii) \implies (ii).
Suppose $i : X \to Z$ is in $\mathcal{C} \cap \mathcal{W}$; then there must be $j : X \to Y$ in $\mathcal{C}'$ and $q : Y \to Z$ in $\mathcal{F}'$ such that $i = q \circ j$, and so we have the commutative diagram shown below:
\[
\begin{tikzcd}
X \dar[swap]{i} \rar{j} &
Y \dar{q} \\
Z \rar[swap]{\id} &
Z
\end{tikzcd}
\]
Since $i \llpwrt q$, $i$ must be a retract of $j$; hence, $i$ is in $\mathcal{C}'$, and therefore $\mathcal{C} \cap \mathcal{W} \subseteq \mathcal{C}'$.

\medskip\noindent
(ii) \implies (i).
See Lemma~14.2.5 in \citep{May-Ponto:2012}.
\end{proof}

\begin{thm}
\label{thm:completeness.for.strongly.accessible.model.categories}
Let $\tuple{L', R}$ and $\tuple{L, R'}$ be functorial weak factorisation systems on a locally presentable category $\mathcal{M}$ and let $\mathcal{F}$ and $\mathcal{F}'$ be the full subcategories of $\Func{\mathbf{2}}{\mathcal{M}}$ spanned by the morphisms in the right class of of the weak factorisation systems induced by $\tuple{L', R}$ and $\tuple{L, R'}$, respectively.
Suppose $\kappa$ and $\lambda$ are regular cardinals  satisfying the following hypotheses:
\begin{itemize}
\item $\mathcal{M}$ is a locally $\kappa$-presentable category, and $\kappa \sharplylt \lambda$.

\item $\mathcal{F}$ and $\mathcal{F}'$ are closed under colimits of small $\kappa$-filtered diagrams in $\Func{\mathbf{2}}{\mathcal{M}}$.

\item $R, R' : \Func{\mathbf{2}}{\mathcal{M}} \to \Func{\mathbf{2}}{\mathcal{M}}$ are both $\kappa$-accessible and strongly $\lambda$-accessible.
\end{itemize}
Let $\mathcal{C}'$ be the full subcategory of $\Func{\mathbf{2}}{\mathcal{M}}$ spanned by the morphisms in the left class of the weak factorisation system induced by $\tuple{L', R}$ and let $\mathcal{W}$ be the preimage of $\mathcal{F}'$ under the functor $R : \Func{\mathbf{2}}{\mathcal{M}} \to \Func{\mathbf{2}}{\mathcal{M}}$.
Then:
\begin{enumerate}[(i)]
\item The functorial weak factorisation systems $\tuple{L', R}$ and $\tuple{L, R'}$ restrict to functorial weak factorisation systems on $\Kompakt[\lambda]{\mathcal{M}}$.

\item The inclusions $\mathcal{F} \embedinto \Func{\mathbf{2}}{\mathcal{M}}$ and $\mathcal{F}' \embedinto \Func{\mathbf{2}}{\mathcal{M}}$ are strongly $\lambda$-accessible functors.

\item $\mathcal{W}$ is closed under colimits of small $\kappa$-filtered diagrams in $\Func{\mathbf{2}}{\mathcal{M}}$, and the inclusion $\mathcal{W} \embedinto \Func{\mathbf{2}}{\mathcal{M}}$ is a strongly $\lambda$-accessible functor.

\item $\mathcal{C}' \subseteq \mathcal{W}$ if and only if the same holds in $\Kompakt[\lambda]{\mathcal{M}}$.

\item $\mathcal{F}' = \mathcal{W} \cap \mathcal{F}$ if and only if the same holds in $\Kompakt[\lambda]{\mathcal{M}}$.

\item $\mathcal{W}$ (regarded as a class of morphisms in $\mathcal{M}$) has the 2-out-of-3 property in $\mathcal{M}$ if and only if the same is true in $\Kompakt[\lambda]{\mathcal{M}}$.

\item The weak factorisation systems induced by $\tuple{L', R}$ and $\tuple{L, R'}$ underlie a model structure on $\mathcal{M}$ if and only if their restrictions to $\Kompakt[\lambda]{\mathcal{M}}$ underlie a model structure on $\Kompakt[\lambda]{\mathcal{M}}$.
\end{enumerate}
\end{thm}
\begin{proof}
(i).
It is clear that we can restrict $\tuple{L', R}$ and $\tuple{L, R'}$ to obtain functorial factorisation systems on $\Kompakt[\lambda]{\mathcal{M}}$, and these are  functorial \emph{weak} factorisation systems by \autoref{thm:functorial.weak.factorisation.systems}.

\medskip\noindent
(ii).
Since $R, R' : \Func{\mathbf{2}}{\mathcal{M}} \to \Func{\mathbf{2}}{\mathcal{M}}$ are both $\kappa$-accessible and strongly $\lambda$-accessible, we may use \autoref{prop:accessibility.of.right.class.of.accessible.fwfs} to deduce that the inclusions $\mathcal{F} \embedinto \Func{\mathbf{2}}{\mathcal{M}}$ and $\mathcal{F}' \embedinto \Func{\mathbf{2}}{\mathcal{M}}$ are strongly $\lambda$-accessible.

\medskip\noindent
(iii).
Since $\mathcal{F}'$ is a replete full subcategory of $\Func{\mathbf{2}}{\mathcal{M}}$, we may use \autoref{prop:preimage.of.accessible.replete.full.subcategories} to deduce that $\mathcal{W}$ is closed under colimits of small $\kappa$-filtered diagrams in $\Func{\mathbf{2}}{\mathcal{M}}$ and that the inclusion $\mathcal{W} \embedinto \Func{\mathbf{2}}{\mathcal{M}}$ is a strongly $\lambda$-accessible functor.

\medskip\noindent
(iv).
The endofunctor $L' : \Func{\mathbf{2}}{\mathcal{M}} \to \Func{\mathbf{2}}{\mathcal{M}}$ is strongly $\lambda$-accessible, and $\mathcal{W}$ is closed under colimits of small $\lambda$-filtered diagrams, so (recalling propositions~\ref{prop:locally.presentable.functor.categories} and~\ref{prop:presentable.objects.in.diagram.categories}) if $L'$ sends the subcategory $\Func{\mathbf{2}}{\Kompakt[\lambda]{\mathcal{M}}}$ to $\mathcal{W}$, then the entirety of the image of $L'$ must be contained in $\mathcal{W}$.
\Autoref{prop:algebras.and.coalgebras.for.fwfs} implies every object in $\mathcal{C}'$ is a retract of an object in the image of $L'$, and (iii) implies $\mathcal{W}$ is closed under retracts, so we may deduce that $\mathcal{C}' \subseteq \mathcal{W}$ if and only if $\mathcal{C}' \cap \Func{\mathbf{2}}{\Kompakt[\lambda]{\mathcal{M}}} \subseteq \mathcal{W} \cap \Func{\mathbf{2}}{\Kompakt[\lambda]{\mathcal{M}}}$.

\medskip\noindent
(v).
Claims (ii) and (iii) and \autoref{prop:preimage.of.accessible.replete.full.subcategories} imply the inclusion $\mathcal{W} \cap \mathcal{F} \embedinto \Func{\mathbf{2}}{\mathcal{M}}$ is strongly $\lambda$-accessible; but by propositions~\ref{prop:presentable.objects.in.diagram.categories} and~\ref{prop:strongly.accessible.subcategories},
\begin{align*}
\Kompakt[\lambda]{\mathcal{F}'} & = \mathcal{F}' \cap \Func{\mathbf{2}}{\Kompakt[\lambda]{\mathcal{M}}} &
\Kompakt[\lambda]{\mathcal{W} \cap \mathcal{F}} & = \parens{\mathcal{W} \cap \mathcal{F}} \cap \Func{\mathbf{2}}{\Kompakt[\lambda]{\mathcal{M}}}
\end{align*}
so $\mathcal{F}' = \mathcal{W} \cap \mathcal{F}$ if and only if $\mathcal{F}' \cap \Func{\mathbf{2}}{\Kompakt[\lambda]{\mathcal{M}}} = \parens{\mathcal{W} \cap \mathcal{F}} \cap \Func{\mathbf{2}}{\Kompakt[\lambda]{\mathcal{M}}}$.

\medskip\noindent
(vi).
Consider the three full subcategories $\Lambda^2_i \argp{\mathcal{W}}$ (where $i \in \set{0, 1, 2}$) of $\Func{\mathbf{3}}{\mathcal{M}}$ spanned (respectively) by the diagrams of the form below:
\[
\mmhfill
\begin{tikzcd}
\bullet \drar[swap]{\in \mathcal{W}} \rar{\in \mathcal{W}} &
\bullet \dar \\
&
\bullet
\end{tikzcd}
\mmhfill
\begin{tikzcd}
\bullet \drar \rar{\in \mathcal{W}} &
\bullet \dar{\in \mathcal{W}} \\
&
\bullet
\end{tikzcd}
\mmhfill
\begin{tikzcd}
\bullet \drar[swap]{\in \mathcal{W}} \rar &
\bullet \dar{\in \mathcal{W}} \\
&
\bullet
\end{tikzcd}
\mmhfill
\]
By \autoref{prop:products.of.accessible.categories}, each inclusion $\Lambda^2_i \argp{\mathcal{W}} \embedinto \Func{\mathbf{3}}{\mathcal{M}}$ is the pullback of a strongly $\lambda$-accessible inclusion of a full subcategory of $\Func{\mathbf{2}}{\mathcal{M}}^{\times 3}$ along the evident projection functor $\Func{\mathbf{3}}{\mathcal{M}} \to \Func{\mathbf{2}}{\mathcal{M}}^{\times 3}$; thus, each inclusion $\Lambda^2_i \argp{\mathcal{W}} \embedinto \Func{\mathbf{3}}{\mathcal{M}}$ is a strongly $\lambda$-accessible functor.
We may then use \autoref{prop:strongly.accessible.subcategories} as above to prove the claim.

\medskip\noindent
(vii).
Apply \autoref{lem:model.structure.criteria}.
\end{proof}

\begin{cor}
\label{cor:accessible.weak.equivalences}
Let $\mathcal{M}$ be a strongly $\tuple{\kappa, \lambda}$-combinatorial model category.
Then the full subcategory $\mathcal{W}$ of $\Func{\mathbf{2}}{\mathcal{M}}$ spanned by the weak equivalences is closed under colimits of small $\kappa$-filtered diagrams in $\Func{\mathbf{2}}{\mathcal{M}}$, and the inclusion $\mathcal{W} \embedinto \Func{\mathbf{2}}{\mathcal{M}}$ is a strongly $\lambda$-accessible functor.
\end{cor}
\begin{proof}
Combine \autoref{prop:accessibility.of.fibrations.and.trivial.fibrations} and \autoref{thm:completeness.for.strongly.accessible.model.categories}.
\end{proof}

\Autoref{thm:completeness.for.strongly.accessible.model.categories} suggests that free $\lambda$-ind-completions of suitable small model categories are combinatorial model categories.
More precisely:

\begin{dfn}
Let $\kappa$ and $\lambda$ be regular cardinals.
A \strong{$\tuple{\kappa, \lambda}$-miniature model category} is a model category $\mathcal{M}$ that satisfies these axioms:
\begin{itemize}
\item $\mathcal{M}$ is a $\tuple{\kappa, \lambda}$-accessible generated category, and $\kappa \sharplylt \lambda$.

\item $\mathcal{M}$ has limits for finite diagrams and colimits of $\lambda$-small diagrams.

\item Each hom-set in $\Kompakt[\kappa][\lambda]{\mathcal{M}}$ is $\lambda$-small.

\item There exist $\lambda$-small sets of morphisms in $\Kompakt[\kappa][\lambda]{\mathcal{M}}$ that cofibrantly generate the model structure of $\mathcal{M}$.
\end{itemize}
\end{dfn}

\begin{prop}
\label{prop:heart.of.a.combinatorial.model.category}
If $\mathcal{M}$ is a strongly $\tuple{\kappa, \lambda}$-combinatorial model category, then $\Kompakt[\lambda]{\mathcal{M}}$ is a $\tuple{\kappa, \lambda}$-miniature model category (with the weak equivalences, cofibrations, and fibrations inherited from $\mathcal{M}$).
\end{prop}
\begin{proof}
By \autoref{thm:classification.of.accessibly.generated.categories}, $\Kompakt[\lambda]{\mathcal{M}}$ is a $\tuple{\kappa, \lambda}$-accessible generated category, and \autoref{lem:small.objects} implies it is closed under colimits of $\lambda$-small diagrams in $\mathcal{M}$.
Now, choose a pair of functorial factorisation systems as in \autoref{prop:accessibility.of.fibrations.and.trivial.fibrations}, and recall that \autoref{thm:functorial.weak.factorisation.systems} says a morphism is in the left (\resp right) class of a functorial weak factorisation system if and only if it is a retract of the left (\resp right) half of its functorial factorisation.
Since we chose factorisation functors that are strongly $\lambda$-accessible, it follows that the weak factorisation systems on $\mathcal{M}$ restricts to weak factorisation systems on $\Kompakt[\lambda]{\mathcal{M}}$.
It is then clear that $\Kompakt[\lambda]{\mathcal{M}}$
inherits a model structure from $\mathcal{M}$, and \autoref{lem:solution.set.condition.and.lifting.conditions} implies the model structure on $\Kompakt[\lambda]{\mathcal{M}}$ can be cofibrantly generated by $\lambda$-small sets of morphisms in $\Kompakt[\kappa]{\mathcal{M}}$.
The remaining axioms for a $\lambda$-miniature model category are easily verified.
\end{proof}

\begin{remark}
The subcategory $\Kompakt[\lambda]{\mathcal{M}}$ inherits much of the homotopy-theoretic structure of $\mathcal{M}$.
For instance, $\Kompakt[\lambda]{\mathcal{M}}$ has simplicial and cosimplicial resolutions and the inclusion $\Kompakt[\lambda]{\mathcal{M}} \embedinto \mathcal{M}$ preserves them, so the induced $\parens{\Ho \cat{\SSet}}$-enriched functor $\Ho \Kompakt[\lambda]{\mathcal{M}} \to \Ho \mathcal{M}$ is fully faithful, where the $\parens{\Ho \cat{\SSet}}$-enrichment is defined as in \citep[\Chap 5]{Hovey:1999}.
In particular, the induced functor between the ordinary homotopy categories is fully faithful.
\end{remark}

\begin{thm}
\label{thm:model.structure.for.ind-objects}
Let $\mathcal{K}$ be a $\tuple{\kappa, \lambda}$-miniature model category, let $\mathcal{M}$ be the free $\lambda$-ind-completion $\Ind[\lambda]{\mathcal{K}}$, and let $\gamma : \mathcal{K} \to \mathcal{M}$ be the canonical embedding.
\begin{enumerate}[(i)]
\item There is a unique way of making $\mathcal{M}$ into a strongly $\tuple{\kappa, \lambda}$-combinatorial model category such that $\gamma : \mathcal{K} \to \mathcal{M}$ preserves and reflects the model structure.

\item Moreover, for any model category $\mathcal{N}$ with colimits of all small diagrams, restriction along $\gamma : \mathcal{K} \to \mathcal{M}$ induces a functor
\begin{itemize}
\item \emph{from} the full subcategory of $\Func{\mathcal{M}}{\mathcal{N}}$ spanned by the left Quillen functors

\item \emph{to} the full subcategory of $\Func{\mathcal{K}}{\mathcal{N}}$ spanned by the functors that preserve cofibrations, trivial cofibrations, and colimits of $\lambda$-small diagrams.
\end{itemize} 
\end{enumerate}
\end{thm}
\begin{proof}
(i).
We will identify $\mathcal{K}$ with the image of $\gamma : \mathcal{K} \to \mathcal{M}$.
Note that $\mathcal{M}$ is a locally $\kappa$-presentable category, by \autoref{thm:classification.of.accessibly.generated.categories}.
Let $\mathcal{I}$ (\resp $\mathcal{I}'$) be a $\lambda$-small set of morphisms in $\Kompakt[\kappa][\lambda]{\mathcal{K}}$ that generate the cofibrations (\resp trivial cofibrations) in $\mathcal{K}$.
Let $\tuple{L', R}$ and $\tuple{L, R'}$ be functorial weak factorisation systems cofibrantly generated by $\mathcal{I}'$ and $\mathcal{I}$ respectively; by \autoref{thm:Quillen.small.object.argument}, we may assume $R, R' : \Func{\mathbf{2}}{\mathcal{M}} \to \Func{\mathbf{2}}{\mathcal{M}}$ preserve colimits of small $\kappa$-filtered diagrams and are strongly $\lambda$-accessible functors.

Let $\mathcal{F}$ and $\mathcal{F}'$ be the full subcategories of $\Func{\mathbf{2}}{\mathcal{M}}$ spanned by the right class of the weak factorisation systems induced by $\tuple{L', R}$ and $\tuple{L, R'}$, respectively.
It is not hard to see that any morphism in $\mathcal{K}$ is an object in $\mathcal{F}$ (\resp $\mathcal{F}'$) if and only if it is a fibration (\resp trivial fibration) in $\mathcal{K}$.
\Autoref{lem:filtered.colimits.in.the.right.class.of.cofibrantly-generated.wfs} says $\mathcal{F}$ and $\mathcal{F}'$ are closed under colimits of small $\kappa$-filtered diagrams in $\Func{\mathbf{2}}{\mathcal{M}}$, so we may now apply \autoref{thm:completeness.for.strongly.accessible.model.categories} to deduce that $\mathcal{F}$ and $\mathcal{F}'$ induce a model structure on $\mathcal{M}$.
It is clear that $\mathcal{M}$ equipped with this model structure is then a strongly $\tuple{\kappa, \lambda}$-combinatorial model category in a way that is compatible with the canonical embedding $\mathcal{K} \to \mathcal{M}$.

Finally, to see that the above construction is the unique way of making $\mathcal{M}$ into a strongly $\tuple{\kappa, \lambda}$-combinatorial model category satisfying the given conditions, we simply have to observe that the model structure of a strongly $\tuple{\kappa, \lambda}$-\allowhyphens combinatorial model category is necessarily cofibrantly generated by the cofibrations and trivial cofibrations in (a small skeleton of) $\Kompakt[\kappa]{\mathcal{M}}$ (independently of the choice of $\mathcal{I}$ and $\mathcal{I}'$).

\medskip\noindent
(ii).
Clearly, every left Quillen functor $F : \mathcal{M} \to \mathcal{N}$ restricts to a functor $F \gamma : \mathcal{K} \to \mathcal{N}$ that preserves cofibrations, trivial cofibrations, and colimits of $\lambda$-small diagrams.
Conversely, given any such functor $F' : \mathcal{K} \to \mathcal{N}$, we may apply \autoref{thm:universal.property.of.free.ind.completion} to obtain a $\lambda$-accessible functor $F : \mathcal{M} \to \mathcal{N}$ such that $F \gamma = F'$.
Since cofibrations and trivial cofibrations in $\mathcal{M}$ are generated under colimits of $\lambda$-filtered diagrams by cofibrations and trivial cofibrations in $\mathcal{K}$, the functor $F : \mathcal{M} \to \mathcal{N}$ preserves cofibrations and trivial cofibrations if $F' : \mathcal{K} \to \mathcal{N}$ does.
A similar argument (using \autoref{prop:locally.presentable.functor.categories}) shows that $F : \mathcal{M} \to \mathcal{N}$ preserves colimits of $\lambda$-small diagrams.
Thus, $F : \mathcal{M} \to \mathcal{N}$ preserves colimits of all small diagrams,
\unskip\footnote{See Lemma~2.25 in \citep{Low:2013a}.}
so it has a right adjoint (by \eg the special adjoint functor theorem) and is indeed a left Quillen functor.
\end{proof}

\begin{remark}
Let $\mbfU$ and $\mbfUplus$ be universes, with $\mbfU \in \mbfUplus$, let $\mathcal{M}$ be a strongly $\tuple{\kappa, \lambda}$-combinatorial model $\mbfU$-category, and let $\mathcal{M} \embedinto \succof{\mathcal{M}}$ be a $\tuple{\kappa, \mbfU, \mbfUplus}$-\allowhyphens extension in the sense of \citep{Low:2013a}.
By combining  \autoref{prop:heart.of.a.combinatorial.model.category} and \autoref{thm:model.structure.for.ind-objects}, we may deduce that there is a unique way of making $\succof{\mathcal{M}}$ into a strongly $\tuple{\kappa, \lambda}$-\allowhyphens combinatorial model $\mbfUplus$-category such that the embedding $\mathcal{M} \embedinto \succof{\mathcal{M}}$ preserves and reflects the model structure.
In view of \autoref{prop:combinatorial.model.categories.are.strongly.combinatorial}, it follows that every combinatorial model $\mbfU$-category can be canonically extended to a combinatorial model $\mbfUplus$-category; moreover, by Theorem~3.11 in \opcit, the extension does not depend on $\tuple{\kappa, \lambda}$.
\end{remark}

The techniques used in the proof of \autoref{thm:completeness.for.strongly.accessible.model.categories} are easily generalised to combinatorial model categories with desirable properties.

\begin{thm}
\label{thm:completeness.for.right.properness}
Let $\mathcal{M}$ be a strongly $\tuple{\kappa, \lambda}$-combinatorial model category.
The following are equivalent:
\begin{enumerate}[(i)]
\item $\mathcal{M}$ is a right proper model category.

\item $\Kompakt[\lambda]{\mathcal{M}}$ is a right proper model category.
\end{enumerate}
\end{thm}
\begin{proof}
(i) \implies (ii).
Immediate, because the model structure on $\Kompakt[\lambda]{\mathcal{M}}$ is the restriction of the model structure on $\mathcal{M}$ and $\Kompakt[\lambda]{\mathcal{M}}$ is closed under finite limits in $\mathcal{M}$.

\medskip\noindent
(ii) \implies (i).
Let $\mathcal{D} = \set{ \bullet \rightarrow \bullet \leftarrow \bullet }$, \ie the category freely generated by a cospan.
Since $\mathcal{D}$ is a finite category and $\mathcal{M}$ is a locally $\kappa$-presentable category, \autoref{prop:locally.presentable.functor.categories} says $\Func{\mathcal{D}}{\mathcal{M}}$ is also a locally $\kappa$-presentable category, and \autoref{prop:presentable.objects.in.diagram.categories} implies the $\kappa$-presentable objects in $\Func{\mathcal{D}}{\mathcal{M}}$ are precisely the componentwise $\kappa$-presentable objects.
Thus, the functor $\Delta : \mathcal{M} \to \Func{\mathcal{D}}{\mathcal{M}}$ is strongly $\kappa$-accessible, so \autoref{prop:left.adjoints.and.strong.accessibility} says its right adjoint $\prolim_\mathcal{D} : \Func{\mathcal{D}}{\mathcal{M}} \to \mathcal{M}$ is $\kappa$-accessible; moreover, it is strongly $\lambda$-accessible because $\Kompakt[\lambda]{\mathcal{M}}$ is closed under finite limits in $\mathcal{M}$.

Consider the full subcategory $\mathcal{P} \subseteq \Func{\mathcal{D}}{\mathcal{M}}$ spanned by those diagrams in $\mathcal{M}$ of the form below,
\[
\begin{tikzcd}
{} &
\bullet \dar{w} \\
\bullet \rar[swap]{p} &
\bullet
\end{tikzcd}
\]
where $p$ is a fibration and $w$ is a weak equivalence.
Propositions~\ref{prop:preimage.of.accessible.replete.full.subcategories} and~\ref{prop:accessibility.of.fibrations.and.trivial.fibrations}, \autoref{thm:completeness.for.strongly.accessible.model.categories}, and \autoref{cor:accessible.weak.equivalences} together imply that $\mathcal{P}$ is closed under colimits of small $\kappa$-filtered diagrams in $\Func{\mathcal{D}}{\mathcal{M}}$ and that the inclusion $\mathcal{P} \embedinto \Func{\mathcal{D}}{\mathcal{M}}$ is a strongly $\lambda$-accessible functor.
Since $\prolim_\mathcal{D} : \Func{\mathcal{D}}{\mathcal{M}} \to \mathcal{M}$ is strongly $\lambda$-accessible and the class of weak equivalences in $\mathcal{M}$ is closed under $\lambda$-filtered colimits in $\Func{\mathbf{2}}{\mathcal{M}}$, it follows that $\mathcal{M}$ is right proper if $\Kompakt[\lambda]{\mathcal{M}}$ is.
\end{proof}

\begin{remark}
It is tempting to say that the analogous proposition for left properness follows by duality; unfortunately, the opposite of a combinatorial model category is almost never a combinatorial model category! Nonetheless, the main idea in the proof above can be made to work under the assumption that the category of coalgebras for the left half of the functorial (cofibration, trivial fibration)-factorisation system is generated under colimits of small $\lambda$-filtered diagrams of coalgebras whose underlying object in $\Func{\mathbf{2}}{\mathcal{M}}$ is a cofibration in $\Kompakt[\lambda]{\mathcal{M}}$.
It is not clear whether this hypothesis is always satisfied if we only assume that $\mathcal{M}$ is a strongly $\tuple{\kappa, \lambda}$-combinatorial model category, but it is certainly true if $\lambda$ is sufficiently large, because the category of coalgebras for an accessible copointed endofunctor is always accessible (by an analogue of \autoref{thm:strongly.accessible.pointed.endofunctors}) and any accessible functor is strongly $\lambda$-accessible for large enough $\lambda$ (by \autoref{lem:accessible.functors.are.strongly.accessible}).
\end{remark}

\begin{thm}
\label{thm:completeness.for.SM7}
Let $\ul{\mathcal{M}}$ be a locally small simplicially enriched category where the underlying ordinary category $\mathcal{M}$ is equipped with a model structure making it a strongly $\tuple{\kappa, \lambda}$-combinatorial model category.
Assuming the simplicially enriched full subcategory $\Kompakt[\lambda]{\ul{\mathcal{M}}} \subseteq \ul{\mathcal{M}}$ determined by $\Kompakt[\lambda]{\mathcal{M}}$ is closed under cotensor products with finite simplicial sets in $\ul{\mathcal{M}}$, the following are equivalent:
\begin{enumerate}[(i)]
\item $\ul{\mathcal{M}}$ is a simplicial model category.

\item The model structure of $\Kompakt[\lambda]{\ul{\mathcal{M}}}$ satisfies axiom {\LiningNumbers SM7}.
\end{enumerate} 
\end{thm}
\begin{proof}
(i) \implies (ii).
Immediate, because the model structure of $\Kompakt[\lambda]{\mathcal{M}}$ is the restriction of the model structure of $\mathcal{M}$.

\medskip\noindent
(ii) \implies (i).
Recalling the fact that $\cat{\SSet}$ is a strongly $\tuple{\aleph_0, \aleph_1}$-combinatorial model category, this is a consequence of  propositions \ref{prop:scats:SM7ab} and \ref{prop:accessibility.of.fibrations.and.trivial.fibrations}.
\end{proof}

\begin{remark}
In view of of the above theorem, it should seem very likely that the free $\lambda$-ind-completion of a suitable small simplicial model category will again be a simplicial model category.
To prove this, we require the technology of enriched accessibility introduced by \citet{Kelly:1982} and \citet{Borceux-Quinteriro:1996}; in fact, the only thing we need is to show that the free $\lambda$-ind-completion of a $\lambda$-cocomplete $\cat{\SSet}$-enriched category is a cocomplete $\cat{\SSet}$-enriched category, and this can be done by mimicking the proof for the case of ordinary categories.
The details are left to the reader.
\end{remark}

\appendix

\section{Accessibility}
\label{sect:accessibility}

\needspace{3\baselineskip}
To avoid confusion, let us begin by recalling some basic terminology.

\begin{dfn}
A \strong{regular cardinal} is an infinite cardinal $\kappa$ with the following property:
\begin{itemize}
\item If $\Phi$ is a set of cardinality $< \kappa$ and each element of $\Phi$ is a set of cardinality $< \kappa$, then $\bigcup_{X \in \Phi} X$ is also a set of cardinality $< \kappa$.
\end{itemize}
\end{dfn}

\begin{numpar}
Throughout this section, $\kappa$ is an arbitrary regular cardinal.
\end{numpar}

\begin{dfn}
\ \noprelistbreak
\begin{itemize}
\item A \strong{$\kappa$-small set} is a set of cardinality $< \kappa$.

\item A \strong{$\kappa$-small category} is a category with $< \kappa$ morphisms.

\item A \strong{$\kappa$-small diagram} is a functor whose domain is a $\kappa$-small category.
\end{itemize}
\end{dfn}

\begin{dfn}
\ \noprelistbreak
\begin{itemize}
\item A \strong{$\kappa$-filtered category} is a category $\mathcal{J}$ with the following property: 
\begin{itemize}
\item For each $\kappa$-small diagram $A$ in $\mathcal{J}$, there exist an object $j$ and a cocone $A \hoto \Delta j$.
\end{itemize}
A \strong{$\kappa$-filtered diagram} in a category $\mathcal{C}$ is a functor $\mathcal{J} \to \mathcal{C}$ where $\mathcal{J}$ is a $\kappa$-filtered category.

\item  A \strong{$\kappa$-directed preorder} is a preordered set $X$ that is $\kappa$-filtered when considered as a category, \ie a preorder with the following property:
\begin{itemize}
\item For each $\kappa$-small subset $Y \subseteq X$, there exists an element $x$ of $X$ such that $y \le x$ for all $y$ in $Y$.
\end{itemize}
A \strong{$\kappa$-directed diagram}  in a category $\mathcal{C}$ is a functor $\mathcal{J} \to \mathcal{C}$ where $\mathcal{J}$ is a $\kappa$-directed preorder (considered as a category).
\end{itemize}
It is conventional to say `filtered' (\resp `directed') instead of `$\aleph_0$-filtered' (\resp `$\aleph_0$-directed').
\end{dfn}

\begin{dfn}
\ \noprelistbreak
\begin{itemize}
\item A \strong{cofinal functor} is a functor $F : \mathcal{I} \to \mathcal{J}$ such that, for every object $j$ in $\mathcal{J}$, the comma category $\commacat{j}{F}$ is connected.

\item A \strong{cofinal subcategory} is a subcategory such that the inclusion is a cofinal functor.
\end{itemize}
\end{dfn}

\begin{lem}
\label{lem:very.small.filtered.categories}
If $\mathcal{J}$ is a $\kappa$-small $\kappa$-filtered category, then there exist an object $j$ in $\mathcal{J}$ and an idempotent endomorphism $e : j \to j$ such that the subcategory of $\mathcal{J}$ generated by $e$ is cofinal in $\mathcal{J}$.
\end{lem}
\begin{proof}
Since $\id : \mathcal{J} \to \mathcal{J}$ is a $\kappa$-small diagram in $\mathcal{J}$, there must exist an object $j$ in $\mathcal{J}$ and a cocone $\lambda : \id \hoto \Delta j$.
Let $e = \lambda_j : j \to j$.
Since $\lambda$ is a cocone, we must have $e = e \circ e$, \ie $e : j \to j$ is idempotent.

Let $\mathcal{I}$ be the subcategory of $\mathcal{J}$ generated by $e$ and let $j'$ be any object in $\mathcal{J}$.
We must show that the comma category $\commacat{j'}{\mathcal{I}}$ is connected.
It is inhabited: $\lambda_{j'} : j' \to j$ is an object in $\commacat{j'}{\mathcal{I}}$.
Moreover, given any morphism $f : j' \to j$ in $\mathcal{J}$, we must have $\lambda_{j'} = \lambda_j \circ f = e \circ f$, so $\commacat{j'}{\mathcal{I}}$ is indeed connected.
Thus, $\mathcal{I}$ is a cofinal subcategory of $\mathcal{J}$.
\end{proof}

\begin{lem}[Products of filtered categories]
\label{lem:products.of.filtered.categories}
Let $\seq{ \mathcal{J}_i }{ i \in I }$ be a set of $\kappa$-filtered categories.
\begin{enumerate}[(i)]
\item The product $\mathcal{J} = \prod_{i \in I} \mathcal{J}_i$ is a $\kappa$-filtered category.

\item Each projection $\pi_i : \mathcal{J} \to \mathcal{J}_i$ is a cofinal functor.
\end{enumerate}
\end{lem}
\begin{proof}
(i).
We may construct cones over $\kappa$-small diagrams in $\mathcal{J}$ componentwise.

\medskip\noindent
(ii).
Similarly, one can show that the comma categories $\commacat{j_i}{\pi_i}$ are connected for all $j_i$ in $\mathcal{J}_i$ and all $i$ in $I$.
\end{proof}

\begin{dfn}
Let $\alpha$ be an ordinal.
An \strong{$\alpha$-chain} in a category $\mathcal{C}$ is a functor $\alpha \to \mathcal{C}$, where we have identified $\alpha$ with the well-ordered set of ordinals $< \alpha$.
\end{dfn}

\begin{remark}
If $\alpha$ is an ordinal with cofinality $\kappa$, then $\alpha$ is a $\kappa$-directed preorder.
In particular, $\alpha$-chains are $\kappa$-directed diagrams.
\end{remark}

\begin{dfn}
A \strong{$\kappa$-accessible category} is a locally small category $\mathcal{C}$ satisfying the following conditions:
\begin{itemize}
\item $\mathcal{C}$ has colimits of all small $\kappa$-filtered diagrams.

\item There is a set $\mathcal{G}$ of $\kappa$-presentable objects in $\mathcal{C}$ such that, for each object $B$ in $\mathcal{C}$, there is a small $\kappa$-filtered diagram in $\mathcal{C}$ whose vertices are in $\mathcal{G}$ and whose colimit is $B$.
\end{itemize}
A \strong{locally $\kappa$-presentable category} is a $\kappa$-accessible category that is also cocomplete.

An \strong{accessible category} (\resp \strong{locally presentable category}) is a category that is $\kappa$-accessible (\resp locally $\kappa$-presentable) for some regular cardinal $\kappa$.
\end{dfn}

\begin{dfn}
A \strong{$\kappa$-accessible functor} is a functor $F : \mathcal{C} \to \mathcal{D}$ where $F$ preserves colimits of small $\kappa$-filtered diagrams and $\mathcal{C}$ is a $\kappa$-accessible category.

An \strong{accessible functor} is a functor that is $\kappa$-accessible for some regular cardinal $\kappa$.
\end{dfn}

\begin{thm}
\label{thm:universal.property.of.free.ind.completion}
Let $\mathcal{B}$ be an essentially small category and let $\kappa$ be a regular cardinal.
There exist a $\kappa$-accessible category $\Ind[\kappa]{\mathcal{B}}$ and a functor $\gamma : \mathcal{B} \to \Ind[\kappa]{\mathcal{B}}$ with the following universal property:
\begin{itemize}
\item For any $\kappa$-accessible category $\mathcal{D}$, the induced functor
\[
\gamma^* : \AccFun[\kappa]{\Ind[\kappa]{\mathcal{B}}}{\mathcal{D}} \to \Func{\mathcal{B}}{\mathcal{D}}
\]
is fully faithful and surjective on objects, where $\AccFun[\kappa]{\Ind[\kappa]{\mathcal{B}}}{\mathcal{D}}$ denotes the full subcategory of $\Func{\Ind[\kappa]{\mathcal{B}}}{\mathcal{D}}$ spanned by the $\kappa$-accessible functors.
\end{itemize}
Moreover, the functor $\gamma : \mathcal{B} \to \Ind[\kappa]{\mathcal{B}}$ is fully faithful and injective on objects.
This is the \strong{free $\kappa$-ind-completion of $\mathcal{B}$}.
\end{thm}
\begin{proof} \openproof
See Theorem 2.26 in \citep{LPAC}.
\end{proof}

\begin{prop}
\label{prop:dense.generators.for.accessible.categories}
If $\mathcal{C}$ is a locally small $\kappa$-accessible category, then the Yoneda representation
\[
\mathcal{C} \to \Func{\op{\Kompakt[\kappa]{\mathcal{C}}}}{\cat{\Set}}
\]
is a $\kappa$-accessible fully faithful functor.
\end{prop}
\begin{proof} \openproof
See Proposition~2.1.8 in \citep{Makkai-Pare:1989} or Proposition~2.8 in \citep{LPAC}.
\end{proof}

\begin{prop}
\label{prop:left.adjoints.and.strong.accessibility}
Let $\kappa$ and $\lambda$ be regular cardinals, with $\kappa \le \lambda$, let $\mathcal{C}$ be a $\kappa$-accessible category, and let $\mathcal{D}$ be any category.
Given an adjunction of the form below,
\[
F \dashv G : \mathcal{D} \to \mathcal{C}
\]
the following are equivalent:
\begin{enumerate}[(i)]
\item $F : \mathcal{C} \to \mathcal{D}$ sends $\kappa$-presentable objects in $\mathcal{C}$ to $\lambda$-presentable objects in $\mathcal{D}$.

\item $G : \mathcal{D} \to \mathcal{C}$ preserves colimits of small $\lambda$-filtered diagrams.
\end{enumerate}
\end{prop}
\begin{proof} 
(i) \implies (ii).
Given a $\kappa$-presentable object $C$ in $\mathcal{C}$ and a small $\lambda$-filtered diagram $B : \mathcal{J} \to \mathcal{D}$, observe that
\begin{multline*}
\Hom[\mathcal{C}]{C}{\textstyle G \indlim_\mathcal{J} B} 
  \cong \Hom[\mathcal{D}]{F C}{\textstyle \indlim_\mathcal{J} B}
  \cong \textstyle \indlim_\mathcal{J} \Hom[\mathcal{C}]{F C}{B} \\ 
  \cong \textstyle \indlim_\mathcal{J} \Hom[\mathcal{C}]{C}{G B}
  \cong \Hom[\mathcal{C}]{C}{\textstyle \indlim_\mathcal{J} G B}
\end{multline*}
because $F C$ is a $\lambda$-presentable object in $\mathcal{D}$; but $\kappa$-accessibility of $\mathcal{C}$ implies that the Yoneda representation $\mathcal{C} \to \Func{\op{\Kompakt[\kappa]{\mathcal{C}}}}{\cat{\Set}}$ is fully faithful and reflects colimits of small $\kappa$-filtered diagrams, so this is enough to conclude that $G$ preserves colimits of small $\lambda$-filtered diagrams.

\medskip\noindent
(ii) \implies (i).
Given a $\kappa$-presentable object $C$ in $\mathcal{C}$ and a small $\lambda$-filtered diagram $B : \mathcal{J} \to \mathcal{D}$, observe that
\begin{multline*}
\Hom[\mathcal{D}]{F C}{\textstyle \indlim_\mathcal{J} B}
  \cong \Hom[\mathcal{C}]{C}{\textstyle G \indlim_\mathcal{J} B} 
  \cong \Hom[\mathcal{C}]{C}{\textstyle \indlim_\mathcal{J} G B} \\
  \cong \textstyle \indlim_\mathcal{J} \Hom[\mathcal{C}]{C}{G B} 
  \cong \textstyle \indlim_\mathcal{J} \Hom[\mathcal{C}]{F C}{B}
\end{multline*}
and thus $F C$ is indeed a $\lambda$-presentable object in $\mathcal{D}$.
\end{proof}
\section{Factorisation systems}
\label{sect:factorisation}

\begin{dfn}
A \strong{weak factorisation system} for a category $\mathcal{C}$ is a pair $\tuple{\mathcal{L}, \mathcal{R}}$ of subclasses of $\mor \mathcal{C}$ satisfying these conditions:
\begin{itemize}
\item For each morphism $f$ in $\mathcal{C}$ there exists a pair $\tuple{g, h}$ with $g \in \mathcal{L}$ and $h \in \mathcal{R}$ such that $f = h \circ g$. Such a pair is a \strong{$\tuple{\mathcal{L}, \mathcal{R}}$-factorisation} of $f$.

\item We have $\mathcal{L} = \llpclass{\mathcal{R}}$, \ie a morphism is in $\mathcal{L}$ if and only if it has the left lifting property with respect to every morphism in $\mathcal{R}$.

\item We have $\mathcal{R} = \rlpclass{\mathcal{L}}$, \ie a morphism is in $\mathcal{R}$ if and only if it has the right lifting property with respect to every morphism in $\mathcal{L}$.
\end{itemize}
\end{dfn}

\begin{remark}
Obviously, $\tuple{\mathcal{L}, \mathcal{R}}$ is a weak factorisation system for $\mathcal{C}$ if and only if $\tuple{\op{\mathcal{R}}, \op{\mathcal{L}}}$ is a weak factorisation system for $\op{\mathcal{C}}$.
\end{remark}

\begin{prop}[The retract argument]
\label{prop:weak.factorisation.systems.and.retracts}
Let $\mathcal{C}$ be a category and let $\tuple{\mathcal{L}, \mathcal{R}}$ be a pair of subclasses of $\mor \mathcal{C}$ such that $\mathcal{L} \subseteq \llpclass{\mathcal{R}}$ and $\mathcal{R} \subseteq \rlpclass{\mathcal{L}}$. If every morphism in $\mathcal{C}$ admits an $\tuple{\mathcal{L}, \mathcal{R}}$-factorisation, then the following are equivalent:
\begin{enumerate}[(i)]
\item $\tuple{\mathcal{L}, \mathcal{R}}$ is a weak factorisation system.

\item $\mathcal{L}$ and $\mathcal{R}$ are both closed under retracts in $\mathcal{C}$.
\end{enumerate}
\end{prop}
\begin{proof} \openproof
See Observation 1.3 in \citep{AHRT:2002b}.
\end{proof}

\begin{numpar}
Let $\mathbf{2}$ be the category $\set{ 0 \to 1 }$ and let $\mathbf{3}$ be $\set{ 0 \to 1 \to 2 }$. Thus, given a category $\mathcal{C}$, the functor category $\Func{\mathbf{2}}{\mathcal{C}}$ is the category of arrows and commutative squares in $\mathcal{C}$. There are three embeddings $\delta^0, \delta^1, \delta^2 : \mathbf{2} \to \mathbf{3}$:
\begin{align*}
\delta^0 \argp{0} & = 1 &
\delta^1 \argp{0} & = 0 &
\delta^2 \argp{0} & = 0 \\
\delta^0 \argp{1} & = 2 &
\delta^1 \argp{1} & = 2 &
\delta^2 \argp{1} & = 1
\end{align*}
These then induce (by precomposition) three functors $d_0, d_1, d_2 : \Func{\mathbf{3}}{\mathcal{C}} \to \Func{\mathbf{2}}{\mathcal{C}}$.
\end{numpar}

\begin{dfn}
A \strong{functorial factorisation system} on a category $\mathcal{C}$ is a pair of functors $L, R : \Func{\mathbf{2}}{\mathcal{C}} \to \Func{\mathbf{2}}{\mathcal{C}}$ for which there exists a (necessarily unique) functor $F : \Func{\mathbf{2}}{\mathcal{C}} \to \Func{\mathbf{3}}{\mathcal{C}}$ satisfying the following equations:
\begin{align*}
d_2 F & = L &
d_1 F & = \id_{\Func{\mathbf{2}}{\mathcal{C}}} &
d_0 F & = R
\end{align*}
A \strong{functorial weak factorisation system} on $\mathcal{C}$ is a weak factorisation system $\tuple{\mathcal{L}, \mathcal{R}}$ together with a functorial factorisation system $\tuple{L, R}$ such that $L f \in \mathcal{L}$ and $R f \in \mathcal{R}$ for all morphisms $f$ in $\mathcal{C}$. We will often abuse notation and refer to the functorial factorisation system $\tuple{L, R}$ as a functorial weak factorisation system, omitting mention of the weak factorisation system $\tuple{\mathcal{L}, \mathcal{R}}$.
\end{dfn}

The following characterisation of functorial weak factorisation systems is essentially a generalisation of the retract argument (\autoref{prop:weak.factorisation.systems.and.retracts}).

\begin{thm}
\label{thm:functorial.weak.factorisation.systems}
Let $\tuple{L, R}$ be a functorial factorisation system on a category $\mathcal{C}$. The following are equivalent:
\begin{enumerate}[(i)]
\item For any two morphisms in $\mathcal{C}$, say $h$ and $k$, $L k \llpwrt R h$.

\item $\tuple{\mathcal{L}, \mathcal{R}}$ is an weak factorisation system on $\mathcal{C}$ extending $\tuple{L, R}$, where:
\begin{align*}
\mathcal{L} & = \set{ g \in \mor \mathcal{C} }{ \exists i \in \mor \mathcal{C} . i \circ g = L g \land R g \circ i = \id_{\codom g}  } \\
\mathcal{R} & = \set{ f \in \mor \mathcal{C} }{ \exists r \in \mor \mathcal{C} . f \circ r = R f \land r \circ L f = \id_{\dom f} }
\end{align*}

\item There is a weak factorisation system $\tuple{\mathcal{L}, \mathcal{R}}$ extending $\tuple{L, R}$.
\end{enumerate}
\end{thm}
\begin{proof} \openproof
See Theorem 2.4 in \citep{Rosicky-Tholen:2002}.
\end{proof}

We can rephrase the above theorem in the language of (co)algebras for (co)pointed endofunctors. This will be essential in our proof of \autoref{prop:accessibility.of.right.class.of.accessible.fwfs}.

\begin{prop}
\label{prop:algebras.and.coalgebras.for.fwfs}
Let $\tuple{L, R}$ be a functorial factorisation system on $\mathcal{C}$ and let $\lambda : \id_{\Func{\mathbf{2}}{\mathcal{C}}} \hoto R$ and $\rho : L \hoto \id_{\Func{\mathbf{2}}{\mathcal{C}}}$ be the natural transformations whose component at an object $f$ in $\Func{\mathbf{2}}{\mathcal{C}}$ correspond to the following commutative squares in $\mathcal{C}$:
\[
\mmhfill
\begin{tikzcd}
\bullet \dar[swap]{f} \rar{L f} &
\bullet \dar{R f} \\
\bullet \rar[equals] &
\bullet
\end{tikzcd}
\mmhfill
\begin{tikzcd}
\bullet \dar[swap]{L f} \rar[equals] &
\bullet \dar{f} \\
\bullet \rar[swap]{R f} &
\bullet
\end{tikzcd}
\mmhfill
\]
Suppose $\tuple{L, R}$ extends to a functorial weak factorisation system. Then the following are equivalent for a morphism $g : Z \to W$ in $\mathcal{C}$:
\begin{enumerate}[(i)]
\item The morphism $g$ is in the left class of the induced weak factorisation system.

\item There exists a morphism $i$ in $\mathcal{C}$ such that the diagram below commutes:
\[
\begin{tikzcd}
Z \dar[swap]{g} \rar[equals] &
Z \dar[swap]{L g} \rar[equals] &
Z \dar{g} \\
W \rar{i} \arrow[bend right=20, swap]{rr}{\id} &
\bullet \rar{R g} &
W
\end{tikzcd}
\]

\item The object $g$ in $\Func{\mathbf{2}}{\mathcal{C}}$ admits a coalgebra structure for the copointed endofunctor $\tuple{L, \rho}$.
\end{enumerate}
\needspace{2.5\baselineskip}
Dually, the following are equivalent for a morphism $f : X \to Y$ in $\mathcal{C}$:
\begin{enumerate}[(i\prime)]
\item The morphism $f$ is in the right class of the induced weak factorisation system.

\item There exists a morphism $r$ in $\mathcal{C}$ such that the diagram below commutes:
\[
\begin{tikzcd}
X \dar[swap]{f} \rar[swap]{L f} \arrow[bend left=20]{rr}{\id} &
\bullet \dar{R f} \rar[swap]{r} &
X \dar{f} \\
Y \rar[equals] &
Y \rar[equals] &
Y
\end{tikzcd}
\]

\item The object $f$ in $\Func{\mathbf{2}}{\mathcal{C}}$ admits an algebra structure for the pointed endofunctor $\tuple{R, \lambda}$.
\end{enumerate}
\end{prop}
\begin{proof}
(i) \implies (ii). Consider the following commutative diagram in $\mathcal{C}$:
\[
\begin{tikzcd}
Z \dar[swap]{g} \rar{L g} &
\bullet \dar{R g} \\
W \rar[swap]{\id} &
W
\end{tikzcd}
\]
Thus, a morphism $i$ of the required form exists in $\mathcal{C}$ as soon as $g \llpwrt R g$.

\medskip\noindent
(ii) \iff (iii). This is simply the definition of $\tuple{L, \rho}$-coalgebra.

\medskip\noindent
(ii) \implies (i). By definition, the morphism $L f$ is in the left class of the induced weak factorisation system; but the given diagram exhibits $f$ as a retract of $L f$, so we may apply \autoref{prop:weak.factorisation.systems.and.retracts} to deduce that $f$ is also in the left class.
\end{proof}

\begin{dfn}
A weak factorisation system $\tuple{\mathcal{L}, \mathcal{R}}$ on a category $\mathcal{C}$ is \strong{cofibrantly generated} by a subset $\mathcal{I} \subseteq \mor \mathcal{C}$ if $\mathcal{R} = \rlpclass{\mathcal{I}}$.
\end{dfn}
\section{Model structures}
\label{sect:model.structures}

For the purposes of this paper, it will be convenient to use the following definition of model category:

\begin{dfn}
A \strong{model structure} on a category $\mathcal{M}$ is a triple $\tuple{\mathcal{C}, \mathcal{W}, \mathcal{F}}$ of subclasses of $\mor \mathcal{M}$ satifying the following conditions:
\begin{itemize}
\item $\mathcal{W}$ has the 2-out-of-3 property in $\mathcal{M}$, \ie given a commutative diagram in $\mathcal{M}$ of the form below,
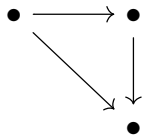
\[
\begin{tikzcd}
\bullet \drar \rar &
\bullet \dar \\
&
\bullet
\end{tikzcd}
\]
if any two of the arrows are in $\mathcal{W}$, then so is the third.

\item $\tuple{\mathcal{C} \cap \mathcal{W}, \mathcal{F}}$ and $\tuple{\mathcal{C}, \mathcal{W} \cap \mathcal{F}}$ are weak factorisation systems on $\mathcal{M}$.
\end{itemize}
Given a model structure $\tuple{\mathcal{C}, \mathcal{W}, \mathcal{F}}$ on a category,
\begin{itemize}
\item a \strong{weak equivalence} is a morphism in $\mathcal{W}$,

\item a \strong{cofibration} is a morphism in $\mathcal{C}$,

\item a \strong{fibration} is a morphism in $\mathcal{F}$,

\item a \strong{trivial cofibration} is a morphism in $\mathcal{C} \cap \mathcal{W}$, and

\item a \strong{trivial fibration} is a morphism in $\mathcal{W} \cap \mathcal{F}$.
\end{itemize}

A \strong{model category} is a locally small category that has limits and colimits for finite diagrams and is equipped with a model structure.
\end{dfn}

\begin{remark}
Let $\mathcal{M}$ be a category. Then, $\tuple{\mathcal{C}, \mathcal{W}, \mathcal{F}}$ is a model structure on $\mathcal{M}$ if and only if $\tuple{\op{\mathcal{F}}, \op{\mathcal{W}}, \op{\mathcal{C}}}$ is a model structure on $\op{\mathcal{M}}$.
\end{remark}

The retract argument (\autoref{prop:weak.factorisation.systems.and.retracts}) shows that model categories in the classical sense satisfy the axioms given above, and for the converse, we require the following fact:

\begin{lem}
The class of weak equivalences in a model category is closed under retracts.
\end{lem}
\begin{proof} \openproof
See Lemma 14.2.5 in \citep{May-Ponto:2012}.
\end{proof}

Moreover, a model structure is completely determined by the two weak factorisation systems:

\begin{lem}
\label{lem:factorising.weak.equivalences}
Let $\mathcal{M}$ be a category equipped with a model structure. The following are equivalent for a morphism $f$ in $\mathcal{M}$:
\begin{enumerate}[(i)]
\item $f$ is a weak equivalence in $\mathcal{M}$.

\item For any factorisation $f = p \circ j$ in $\mathcal{M}$ where $p$ is a fibration and $j$ is a trivial cofibration, $p$ must be a trivial fibration.

\item There exist a trivial cofibration $j$ and a trivial fibration $q$ such that $f = q \circ j$.
\end{enumerate}
\end{lem}
\begin{proof}
(i) \implies (ii). Use the 2-out-of-3 property of weak equivalences.

\medskip\noindent
(ii) \implies (iii). Consider the (trivial cofibration, fibration)-factorisation of $f$.

\medskip\noindent
(iii) \implies (i). Use the 2-out-of-3 property of weak equivalences again.
\end{proof}

\begin{lem}
\label{lem:ternary.weak.factorisation.systems}
Let $\mathcal{M}$ be a category with a pair of weak factorisation systems $\tuple{\mathcal{C}', \mathcal{F}}$ and $\tuple{\mathcal{C}, \mathcal{F}'}$. Assume $\mathcal{W}$ is a subclass of $\mor \mathcal{C}$ satisfying the following condition:
\[
\mathcal{W} \subseteq \set{ q \circ j }{ j \in \mathcal{C}', q \in \mathcal{F}' }
\]
\begin{enumerate}[(i)]
\item $\mathcal{C} \cap \mathcal{W} \subseteq \mathcal{C}'$.

\item If $\mathcal{C}' \subseteq \mathcal{C} \cap \mathcal{W}$, then $\mathcal{F}' \subseteq \mathcal{F}$ and $\mathcal{C} \cap \mathcal{W} = \mathcal{C}'$.
\end{enumerate}
\needspace{2.5\baselineskip}
Dually:
\begin{enumerate}[(i\prime)]
\item $\mathcal{W} \cap \mathcal{F} \subseteq \mathcal{F}'$.

\item If $\mathcal{F}' \subseteq \mathcal{W} \cap \mathcal{F}$, then $\mathcal{C}' \subseteq \mathcal{C}$ and $\mathcal{W} \cap \mathcal{F} = \mathcal{F}'$.
\end{enumerate}
In particular, assuming $\mathcal{C}' \cup \mathcal{F}' \subseteq \mathcal{W}$, we have $\mathcal{C}' = \mathcal{C} \cap \mathcal{W}$ if and only if $\mathcal{F}' = \mathcal{W} \cap \mathcal{F}$.
\end{lem}
\begin{proof}
(i). Suppose $i : X \to Z$ is in $\mathcal{C} \cap \mathcal{W}$; then there must be $j : X \to Y$ in $\mathcal{C}'$ and $q : Y \to Z$ in $\mathcal{F}'$ such that $i = q \circ j$, and so we have the commutative diagram shown below:
\[
\begin{tikzcd}
X \dar[swap]{i} \rar{j} &
Y \dar{q} \\
Z \rar[swap]{\id} &
Z
\end{tikzcd}
\]
Since $i \llpwrt q$, $i$ must be a retract of $j$; hence, by \autoref{prop:weak.factorisation.systems.and.retracts}, $i$ is in $\mathcal{C}'$, and therefore $\mathcal{C} \cap \mathcal{W} \subseteq \mathcal{C}'$.

\medskip\noindent
(ii). If we know $\mathcal{C}' \subseteq \mathcal{C}$, then $\mathcal{F}' \subseteq \mathcal{F}$; and $\mathcal{C}' \subseteq \mathcal{C} \cap \mathcal{W}$, so by (i) it follows that $\mathcal{C}' = \mathcal{C} \cap \mathcal{W}$.
\end{proof}

\needspace{2.5\baselineskip}
The next definition is due to \citet{Smith:1998}:

\begin{dfn}
A \strong{combinatorial model category} is a locally presentable category $\mathcal{M}$ equipped with a cofibrantly generated model structure, \ie there exist subsets $\mathcal{I}$ and $\mathcal{I}'$ of $\mor \mathcal{M}$ such that $\rlpclass{\mathcal{I}}$ is the class of trivial fibrations in $\mathcal{M}$ and $\rlpclass{\mathcal{I}'}$ is the class of fibrations.
\end{dfn}

\begin{remark}
One can use a small object argument (such as \autoref{thm:Quillen.small.object.argument}) to deduce that there are functorial (trivial cofibration, fibration)- and (cofibration, trivial fibration)-\allowhyphens factorisations in a combinatorial model category.
\end{remark}

Finally, let us recall the definition of `simplicial model category':

\begin{dfn}
A \strong{simplicial model structure} on a simplicially enriched category $\ul{\mathcal{M}}$ is a model structure on the underlying ordinary category $\mathcal{M}$ that satisfies the following axiom:
\begin{itemize}[labelwidth=3.0em, leftmargin=!]
\item[\textbf{\LiningNumbers SM7.}] If $i : Z \to W$ is a cofibration in $\mathcal{M}$ and $p : X \to Y$ is a fibration in $\mathcal{M}$, and the square in the diagram below is a pullback square in $\cat{\SSet}$,
\[
\begin{tikzcd}
\ulHom[\mathcal{M}]{W}{X} \arrow[swap, bend right=20]{ddr}{\ulHom[\mathcal{M}]{i}{X}} \arrow[bend left=10]{drr}{\ulHom[\mathcal{M}]{W}{p}} \drar[dashed]{i^* \pbprod p_*} \\
& 
\ulHom[\mathcal{M}]{Z}{X} \times_{\ulHom[\mathcal{M}]{Z}{Y}} \ulHom[\mathcal{M}]{W}{Y} \dar \rar &
\ulHom[\mathcal{M}]{W}{Y} \dar{\ulHom[\mathcal{M}]{i}{Y}} \\
&
\ulHom[\mathcal{M}]{Z}{X} \rar[swap]{\ulHom[\mathcal{M}]{Z}{p}} &
\ulHom[\mathcal{M}]{Z}{Y}
\end{tikzcd}
\]
then the unique morphism $i^* \pbprod p_*$ making the diagram commute is a Kan fibration; moreover, if either $i : Z \to W$ or $p : X \to Y$ is a weak equivalence, then $i^* \pbprod p_*$ is a trivial Kan fibration.
\end{itemize}

A \strong{simplicial model category} is a locally small simplicially enriched category $\ul{\mathcal{M}}$ that has limits and colimits for finite diagrams, tensor and cotensor products with finite simplicial sets, and is equipped with a simplicial model structure.
\end{dfn}

\begin{prop}
\label{prop:scats:SM7ab}
Let $\ul{\mathcal{M}}$ be a locally small simplicially enriched category with limits and colimits for finite diagrams and tensor and cotensor products with finite simplicial sets. Given a model structure on $\mathcal{M}$, the following are equivalent:
\begin{enumerate}[(i)]
\item Axiom {\LiningNumbers SM7} is satisfied.

\item For all fibrations (\resp trivial fibrations) $p : X \to Y$ in $\mathcal{M}$, if $i : Z \to W$ is a boundary inclusion $\partial \Delta^n \embedinto \Delta^n$ and the square in the diagram below is a pullback square in $\mathcal{M}$,
\[
\begin{tikzcd}
W \cotens X \arrow[swap, bend right=20]{ddr}{i \cotens \id_X} \arrow[bend left=10]{drr}{\id_W \cotens p} \drar[dashed]{i \pbprod p} \\
& 
\parens{Z \cotens X} \times_{Z \cotens Y} \parens{W \cotens Y} \dar \rar &
W \cotens Y \dar{i \cotens \id_Y} \\
&
Z \cotens X \rar[swap]{\id_Z \cotens p} &
Z \cotens Y
\end{tikzcd}
\]
where $Z \cotens X$ denotes the cotensor product of $Z$ and $X$, then the unique morphism $i \pbprod p$ making the diagram commute is a fibration (\resp trivial fibration); and for all fibrations $p : X \to Y$ in $\mathcal{M}$, if $i : Z \to W$ is a horn inclusion $\Lambda^n_k \embedinto \Delta^n$, then the morphism $i \pbprod p$ defined as above is a trivial fibration.
\end{enumerate}
\end{prop}
\begin{proof} \openproof
This is an exercise in manipulating partial adjunctions and lifting properties; but see also Proposition 9.3.7 in \citep{Hirschhorn:2003}.
\end{proof}

\ifdraftdoc

\else
  \printbibliography
\fi

\end{document}